\documentclass[a4paper,12pt]{article}
\usepackage{latexsym,amsmath,amsthm,amssymb}
\usepackage{a4wide}
\usepackage{hyperref}
\usepackage{marginnote}
\usepackage{color}
\usepackage{cite}
\hypersetup{
pdftitle={Hardy-Rellich}   
pdfauthor={Van Hoang},
colorlinks = true,
linkcolor = magenta,
citecolor = blue,
}

\theoremstyle{plain}
\newtheorem{theorem}{Theorem}[section]

\newtheorem{proposition}[theorem]{Proposition}
\newtheorem{lemma}[theorem]{Lemma}
\newtheorem{corollary}[theorem]{Corollary}

\theoremstyle{definition}

\theoremstyle{remark}

\renewcommand{\thefootnote}{\arabic{footnote}}
\newcommand{\C}[1]{\ensuremath{{\mathcal C}^{#1}}} 

\def\R{\mathbb R}
\def\C{\mathbb C}

\def\al{\alpha}
\def\Om{\Omega}
\def\de{\delta}
\def\De{\Delta} 

\def\si{\sigma}
\def\lam{\lambda}
\def\ep{\epsilon}
\def\na{\nabla}
\def\pa{\partial}
\def\lt{\left}
\def\rt{\right}
\def\o{\overline}

\def\mG{\mathcal{G}}

\def\i0i{\int_0^\infty}

\def\G{\mathbb G}
\def\mR{\mathcal R}

\numberwithin{equation}{section}

\title{The sharp higher order Hardy--Rellich type inequalities on the homogeneous groups}
\author{Van Hoang Nguyen
\footnote{
Institut de Math\'ematiques de Toulouse, Universit\'e Paul Sabatier, 118 Route de Narbonne, 31062 Toulouse c\'edex 09, France.}
}

\begin{document}
\maketitle


\renewcommand{\thefootnote}{}

\footnote{Email: \href{mailto: Van Hoang Nguyen <van-hoang.nguyen@math.univ-toulouse.fr>}{van-hoang.nguyen@math.univ-toulouse.fr}, \href{mailto: Van Hoang Nguyen <vanhoang0610@yahoo.com>}{vanhoang0610@yahoo.com}}

\footnote{2010 \emph{Mathematics Subject Classification\text}: 26D10, 43A85, 22E30, 43A80.}

\footnote{\emph{Key words and phrases\text}: weighted $L^2-$Hardy--Rellich type inequalities, weighted $L^p-$Hardy--Rellich type inequalities, critical Hardy--Rellich type inequalities, uncertainty principles, homogeneous groups}

\renewcommand{\thefootnote}{\arabic{footnote}}
\setcounter{footnote}{0}

\begin{abstract}
We prove several interesting equalities for the integrals of higher order derivatives on the homogeneous groups. As consequences, we obtain the sharp Hardy--Rellich type inequalities for higher order derivatives including both the subcritical and critical inequalities on the homogeneous groups. We also prove several uncertainty principles on the homogeneous groups. Our results seem to be new even in the case of Euclidean space $\R^n$ and give a simple proof of several classical Hardy--Rellich type inequalities in $\R^n$.
\end{abstract}

\section{Introduction}
The motivation of this paper is to prove several inequalities of Hardy--Rellich type in the setting of the homogeneous groups (the most general class of all nilpotent Lie groups) in the framework of equalities. Our obtained inequalities generalize several well-known Hardy--Rellich type inequalities in the Euclidean space $\R^n$. The Hardy--Rellich type inequalities in $\R^n$ involve the integrals of a function and its derivatives. They appear frequently in various branches of mathematics and provide a useful tool, e.g., in the theory and practice of differential equations, in the theory of approximation etc. Let us start by recalling the classical Hardy inequalities in $\R^n$, $n\geq 3$
\begin{equation}\label{eq:L2HardyE}
\int_{\R^n} |\na f|^2 dx \geq \lt(\frac{n-2}2\rt)^2 \int_{\R^n} \frac{|f|^2}{|x|^2} dx
\end{equation} 
for any function $f\in C_0^\infty(\R^n)$. The constant $(n-2)^2/4$ in \eqref{eq:L2HardyE} is sharp. A similar inequality with the same best constant holds if $\R^n$ is replaced by any domain $\Om$ containing the origin. This inequality plays an important role in many areas such as the spectral theory, the theory of partial differential equations associated to the Laplacian, see e.g., \cite{BEL15,BM97,Davies1998} for reviews of this subject. For interest readers, we refer to \cite{BV97,FT2002} for the improvements of \eqref{eq:L2HardyE} when $\R^n$ is replaced by the bounded domains $\Om$ containing the origin. The $L^p-$version of \eqref{eq:L2HardyE} takes the form
\begin{equation}\label{eq:LpHardyE}
\int_{\R^n} |\na f|^p dx \geq \lt(\frac{n-p}p\rt)^p \int_{\R^n} \frac{|f|^p}{|x|^p} dx,
\end{equation}
for any function $f\in C_0^\infty(\R _n)$ with $n\geq 2$ and $1 < p < n$. Again, the constant $(n-p)^p/p^p$ is sharp in \eqref{eq:LpHardyE}.

In the case $p =n$ the inequality \eqref{eq:LpHardyE} fails for any constant. However, in boundede domains, the following inequality holds
\begin{equation}\label{eq:Hardylog}
\int_{B} |\na f|^n dx \geq \lt(\frac{n-1}n\rt)^n\int_{B} \frac{|f|^n}{|x|^n (1+ \ln \frac1{|x|})^n} dx, \quad f \in C_0^\infty(B)
\end{equation}
for any $n\geq 2$, where $B$ denotes the unit centered ball in $\R^n$ (see, e.g., \cite{ET99}). The constant $(n-1)^n/n^n$ in \eqref{eq:Hardylog} is sharp. It was also shown in \cite{ET99} that \eqref{eq:Hardylog} is equivalent to the critical case of the Sobolev--Lorentz inequality. However, \eqref{eq:Hardylog} is not invariant under the scalings as \eqref{eq:L2HardyE} and \eqref{eq:LpHardyE}. In \cite{II2015}, Ioku and Ishiwata established a scaling version of \eqref{eq:Hardylog} as follows
\begin{equation}\label{eq:IokuIshi}
\lt(\frac{n-1}n\rt)^n\int_{B} \frac{|f|^n}{|x|^n |\ln \frac 1{|x|}|^n} dx \leq \int_{B} \lt|\frac x{|x|}\cdot\na f\rt|^n dx, \quad f\in C_0^\infty(B).
\end{equation}
The constant $(n-1)^n/n^n$ in \eqref{eq:IokuIshi} is sharp. The inequality \eqref{eq:IokuIshi} in bounded domains is also discussed in \cite{II2015}. It is surprise that the critical Hardy inequality \eqref{eq:IokuIshi} is equivalent to the subcritical Hardy inequality \eqref{eq:LpHardyE} in larger dimension spaces (see \cite{ST2017}). A global scaling invariant version of \eqref{eq:Hardylog} was proved in \cite{IIO2016} by Ioku, Ishiwata and Ozawa as follows
\begin{equation}\label{eq:criticalHE}
\lt(\frac{n-1}n\rt)^n\int_{\R^n} \frac{|f -f_R|^n}{|x|^n |\ln \frac R{|x|}|^n} dx \leq \int_{\R^n} \lt|\frac x{|x|}\cdot\na f\rt|^n dx,\quad f\in C_0^\infty(\R^n\setminus\{0\})
\end{equation}
for any $R >0$ with $f_R(x) = f(Rx/|x|)$. Indeed, \eqref{eq:criticalHE} was proved in \cite{IIO2016} with a sharp remainder. Again, the constant $(n-1)^n/ n^n$ in \eqref{eq:criticalHE} is sharp. Note that \eqref{eq:criticalHE} implies \eqref{eq:IokuIshi} by taking $R=1$ for any function supported in $B$.

The Rellich inequality states that for any function $f \in H^2(\R^n)$, $n\geq 5$ 
\begin{equation}\label{eq:RellichE}
\int_{\R^n} |\De f|^2 dx \geq \lt(\frac{n(n-4)}4\rt)^2 \int_{\R^n} \frac{|f|^2}{|x|^4} dx.
\end{equation}
The constant $n^2(n-4)^2/16$ is sharp in \eqref{eq:RellichE}. A similar inequality also holds true in $H^2_0(\Om)$ for any smooth domain $\Om$ of $\R^n$ containing the origin. This inequality was first proved by Rellich \cite{Rellich} for functions $f\in H^2_0(\Om)$ and then was extended for functions $f \in H^2(\Om)\cap H^1_0(\Om)$ by Dold et al. in \cite{DGLV1998}. Davies and Hinz \cite{DH1998} generalized \eqref{eq:RellichE} and shows that for any $p \in (1,n/2)$, it holds 
\begin{equation}\label{eq:LpRellichE}
\int_{\R^n} |\De f|^p dx \geq \lt(\frac{n(p-1)(n-2p)}{p^2}\rt)^p \int_{\R^n} \frac{|f|^p}{|x|^{2p}} dx,\quad f\in C_0^\infty(\R^n\setminus\{0\}).
\end{equation}
The constant $(n(p-1)(n-2p))^p/p^{2p}$ is the best constant in \eqref{eq:LpRellichE}. A weighted version with the sharp constant of \eqref{eq:LpRellichE} is also given in \cite{DH1998}. In \cite{AS2009,GGM2004,TZ2007}, the improvements of \eqref{eq:RellichE} in smooth bounded domains $\Om \subset \R^n$ containing the origin was proved by adding the nonnegative remainder terms. 

The Hardy and Rellich inequalities also was extended to the higher order derivatives with weights in \cite{DH1998}. Let $0 < k < n/2$ be an integer and a function $f \in H^k(\R^n)$. Then if $k=2m$
\begin{equation}\label{eq:higherRellichEeven}
\int_{\R^n} (\De^m f)^2 dx \geq \lt(\prod_{i=0}^{m-1}\frac{(n+4i)(n-4-4i)}4\rt)^2 \int_{\R^n} \frac{|f|^2}{|x|^{4m}}dx.
\end{equation}
If $k = 2m+1$, then 
\begin{equation}\label{eq:higherRellichEodd}
\int_{\R^n} (\nabla \De^m f)^2 dx \geq \lt(\frac{n-2}2\rt)^2\prod_{i=1}^{m-1}\lt(\frac{(n +2+4i)(n-6-4i)}4\rt)^2 \int_{\R^n} \frac{|f|^2}{|x|^{4m+2}}dx.
\end{equation}
Again, the inequalities \eqref{eq:higherRellichEeven} and \eqref{eq:higherRellichEodd} are sharp with the best constants. The improved versions of \eqref{eq:higherRellichEeven} and \eqref{eq:higherRellichEodd} in bounded smooth domains of $\R^n$ was proved by Tertikas and Zographopoulos \cite{TZ2007}. We next state the $L^p-$version with weights of \eqref{eq:higherRellichEeven} and \eqref{eq:higherRellichEodd} which was established in \cite{DH1998}. If $k=2m < n$, $m\geq 1$, $p\in (1, n/(2m))$ and $\al \in (-2(p-1)/p, (n-2mp)/p)$ then
\begin{equation}\label{eq:Lphighereven}
\int_{\R^n} \frac{|\De^m f|^p}{|x|^{\al p}} dx \geq \prod_{i=0}^{m-1} \lt(\frac{(n-2p-(2i+\al)p)(n+ p'(2i+\al))}{pp'}\rt)^p \int_{\R^n} \frac{|f|^p}{|x|^{\al p+ 2mp}} dx
\end{equation}
for any function $f\in C_0^\infty(\R^n\setminus\{0\})$, where $p' =p/(p-1)$ with $p > 1$. It was also prove that if $k = 2m + 1 < n$, $m \geq 1$, $p \in (1, n/(2m+1))$ and $\al \in (-2/p', (n-(2m+1)p)/p)$, then
\begin{align}\label{eq:Lphigherodd}
\int_{\R^n} \frac{|\nabla \De^m f|^p}{|x|^{\al p}} dx &\geq \lt(\frac{n -p -\al}p\rt)^p\prod_{i=0}^{m-1} \lt(\frac{(n-2p-(2i+1+\al)p)(n+ p'(2i+1+\al))}{pp'}\rt)^p \times \notag\\
&\qquad \qquad\qquad\qquad\qquad \times \int_{\R^n} \frac{|f|^p}{|x|^{\al p+ (2m+1)p}} dx
\end{align}
for any function $f\in C_0^\infty(\R^n\setminus\{0\})$. Note that the constants in \eqref{eq:Lphighereven} and \eqref{eq:Lphigherodd} are sharp. Davies and Hinz proved the inequality \eqref{eq:Lphighereven} by iterating the sharp weighted Rellich inequality (i.e., a weighted version of \eqref{eq:LpRellichE}). The inequality \eqref{eq:Lphigherodd} follows from the sharp weighted Hardy inequality (i.e., a weighted version of \eqref{eq:LpHardyE}). We refer the reader to \cite{DH1998} for more detail on the proof of \eqref{eq:Lphighereven} and \eqref{eq:Lphigherodd}. In \cite{Mi}, Mitidieri proposed the simple approaches to prove these Hardy--Rellich type inequalities (including \eqref{eq:L2HardyE}, \eqref{eq:LpHardyE}, \eqref{eq:RellichE}, \eqref{eq:LpRellichE}, \eqref{eq:higherRellichEeven}, \eqref{eq:higherRellichEodd}, \eqref{eq:Lphighereven} and \eqref{eq:Lphigherodd}). The first approach is based on the divergence theorem by choosing a special vector field. The second approach is based on the Rellich-Pohozaev type identities \cite{Mi1}. 

In the critical case $p = n/k$, the critical Rellich type inequalities on bounded domains was proved by Adimurthi and Santra \cite{AS2009}. Let us recall their results here. Let $\Om$ be a bounded domain of $\R^n$ containing the origin. If $n=2p$ then we can find a $R >0$ such that
\begin{equation}\label{eq:ASLp}
\int_{\Om} |\De f|^p dx \geq \frac{(n-2)^{2p}}{n^p} \int_{\Om} \frac{|f|^p}{|x|^n \lt| \ln \frac{R}{|x|}\rt|^p} dx,\qquad f\in C_0^\infty(\Om\setminus\{0\}).
\end{equation} 
It is noted that the constant $(n-2)^{2p}/n^p$ is the best constant in \eqref{eq:ASLp}. In the case $n=4$, it was proved that there exists $R_0>0$ such that for any $R \geq R_0$, it holds
\begin{equation}\label{eq:ASL2}
\int_{\Om} |\De f|^2 dx \geq  \int_{\Om} \frac{|f|^2}{|x|^4 \lt| \ln \frac{R}{|x|}\rt|^2} dx,\qquad f\in C_0^\infty(\Om\setminus\{0\}).
\end{equation}
The inequality \eqref{eq:ASL2} is sharp (see also \cite{AGS2006}). Adimurthi and Santra \cite{AS2009} also proved the critical Rellich type inequalities in $H^k(\R^n)$ with $n=2k$. Let $B$ denote the unit ball centered at the origin in $\R^n$. We can find a $R >0$ such that 
\begin{equation}\label{eq:ASchan}
\int_{B} |\De^m f|^2 dx \geq \frac{n^2}{16}\lt(\frac1{2^{2m-2}}\prod_{i=0}^{m-2}(4i+2)(8m-4i-6)\rt)^2 \int_{B} \frac{|f|^2}{|x|^n \lt|\ln \frac R{|x|}\rt|^2} dx
\end{equation}
if $n =4m$, $m \geq 2$ and
\begin{equation}\label{eq:ASle}
\int_{B} |\na \De^m f|^2 dx \geq \frac{n^2}{16}\lt(\frac1{2^{2m}}\prod_{i=0}^{m-1}(4m-4i-2)(4m+4i+2)\rt)^2 \int_{B} \frac{|f|^2}{|x|^n \lt|\ln \frac R{|x|}\rt|^2} dx
\end{equation}
if $n =4m+ 2$, $m \geq 1$ for any radial function $f \in H^k_0(B)$. Again, the inequalities \eqref{eq:ASchan} and \eqref{eq:ASle} are sharp. Moreover, the stronger versions of \eqref{eq:ASLp}, \eqref{eq:ASchan} and \eqref{eq:ASle} with nonnegative remainder terms was proved in \cite{AS2009}.

Besides these Hardy--Rellich type inequalities, there is a similar Rellich inequality that connects first to second order derivatives. That is, for any $n \geq 5$ and for any function $f\in C_0^\infty(\R^n\setminus\{0\})$ we have
\begin{equation}\label{eq:RonetwoE}
\int_{\R^n} |\De f|^2 dx \geq \frac{n^2}{4} \int_{\R^n} \frac{|\na f|^2}{|x|^2} dx.
\end{equation}
The constant $N^2/4$ is the best constant in \eqref{eq:RonetwoE}. This inequality was proved in \cite{AGS2006} by Adimurthi, Grossi and Santra for radial functions. And its improvement was obtained by Tertilas and Zographopoulos in \cite{TZ2007} together with its weighted version. We also mention here that the $L^p-$version of \eqref{eq:RonetwoE} was proved by Adimurthi and Santra in \cite{AS2009} for only radial functions again.

These Hardy--Rellich type inequalities arise very naturally in the study of singular differential operators. They have been intensively analysed in many different settings. The sharp Hardy and Rellich inequalities was extended to the Riemannian manifolds in \cite{AS2006a,KO2009,KO2013,Grillo,DAD2014,Xia}. They were proved in the setting of the fractional Laplacians in \cite{FMT2013,FLS2008,FS2008,Nguyen2016,H1977,Yafaev} and references therein. The Hardy inequality was proved in the Heisenberg group for $p=2$ by Garofalo and Lanconelli \cite{GL} (see also D'Ambrosio \cite{D'ambrosio}) and for any $p\not=2$ by Niu, Zhang and Wang \cite{NZW}. Further extensions were made by Danielli, Garofalo and Phuc \cite{DGP} on groups of Heisenberg type, by Goldstein and Kombe \cite{GK} on polarisable groups, by Jin and Shen \cite{JS} and Lian \cite{Lian} on Carnot groups together with certain weighted versions. An unified approach to the weighted Hardy inequalities on Carnot groups was recently introduced in \cite{GKY2017}. The higher order Hardy inequalities on the stratified Lie groups was obtained by Ciatti, Cowling and Ricci \cite{CCR} for arbitrary homogeneous quasi-norm but without the sharp constant. For the horizontal Hardy, Rellich, Caffarelli--Kohn--Nirenberg inequalities on stratified groups, we refer the reader to the paper of Ruzhansky and Suragan \cite{RS2017carnot}.


Recently, there is an enourmous work on studying the sharp Hardy and sharp Rellich inequalities on the homogeneous groups (see, e.g., \cite{RS2017,RS2016,RS2017L2,RSY2016,ORS2016,RSY2017,RSY2017a,RSY2017b}). Note that the abelian groups $(\R^n,+)$, the Heisenberg group, homogeneous Carnot groups, straified Lie groups and graded Lie groups are all special cases of the homogeneous groups. Before introducing the recent inequalities obtained on the homogeneous groups, let us remark that the Hardy inequalities \eqref{eq:L2HardyE} and \eqref{eq:LpHardyE} can be sharpned to the inequality
\begin{equation}\label{eq:IIOHardy}
\lt(\frac{n-p}p\rt)^p\int_{\R^n} \frac{|f|^p}{|x|^p}dx \leq \int_{\R^n} \lt|\frac x{|x|}\cdot \na f\rt|^p dx,\quad 1 < p < n.
\end{equation}
The remainder terms for \eqref{eq:IIOHardy} have been analysed by Ioku, Ishiwata and Ozawa \cite{IIO2017}, see also Machihara, Ozawa and Wadade \cite{MOW2017}. An extension of \eqref{eq:IIOHardy} have been extended to the homogeneous groups by Ruzhansky and Suragan \cite{RS2017}. Let $\mathbb G$ be a homogeneous groups of homogeneous dimension $Q$, and let $|\cdot|$ be any homogeneous quasi-norm on $\mathbb G$ (we refer the reader to Section $2$ for further details related to the homogeneous group). Let us define the radial operator $\mathcal R = \mathcal R_{|\cdot|}$ with respect to $|\cdot|$ by
\begin{equation}\label{eq:mathR}
\mathcal R:= \mathcal R_{|\cdot|} = \frac{d}{d|x|}.
\end{equation}
Clearly, when $\G$ is abelian group $(\R^n, +)$ and $|\cdot|$ is Euclidean norm on $\R^n$ then $\mathcal R$ is exactly the radian derivative $\pa_r:=\frac x{|x|} \cdot \na $. In \cite{RS2017}, the following generalized $L^p-$ Hardy inequality was proved
\begin{equation}\label{eq:RSHardy}
\lt(\frac{Q-p}p\rt)^p \int_{\G} \frac{|f|^p}{|x|^p} dx \leq \int_{\G} |\mR f|^p dx, \quad 1 < p < Q,
\end{equation}
for any complex-valued function $f \in C^\infty_0(\G\setminus\{0\})$, where $dx$ denotes the Haar measure on $\G$. The constant $(Q-p)^p/p^p$ in \eqref{eq:RSHardy} is sharp for any quasi-norm $|\cdot|$. Moreover, Ruzhansky and Suragan also obtained the remainder terms in \eqref{eq:RSHardy} and proved an uncertainty principle on $\mG$ from \eqref{eq:RSHardy}. The inequality in the critical case $p=Q$ of \eqref{eq:RSHardy} was discussed in \cite{RS2016}
\begin{equation}\label{eq:criticalGRS}
\lt(\frac{p-1}p\rt)^p \sup_{R>0} \int_{\G} \frac{|f-f_R|^p}{|x|^Q \lt|\ln \frac R{|x|}\rt|^p} dx \leq \int_{\G} \frac{|\mR f|^p}{|x|^{Q-p}} dx,
\end{equation}
for any $1 < p < \infty$, where $f_R = f(Rx/|x|)$. The inequality \eqref{eq:criticalGRS} generalizes the inequality \eqref{eq:criticalHE} to the homogeneous groups for any homogeneous quasi-norm. 

In \cite{RS2017}, Ruzhansky and Suragan also proved a Rellich inequality on $\G$. Denote 
\begin{equation}\label{eq:R2a}
\mR_2 = \mR^2 + \frac{Q-1}{|x|} \mR.
\end{equation}
Then the following inequality holds
\begin{equation}\label{eq:RellichGRS}
\frac{Q^2(Q-4)^2}{16} \int_{\G} \frac{|f|^2}{|x|^4} dx \leq \int_{\G} |\mR_2 f|^2 dx,
\end{equation}
for any complex-valued function $f \in C^\infty_0(\G\setminus\{0\})$. The constant $Q^2(Q-4)^2/16$ is the best constant in \eqref{eq:RellichGRS} for any homogeneous quasi-norm $|\cdot|$. The appearance of $\mathcal R_2$ in \eqref{eq:RellichGRS} is nature since there is no analogue of homogeneous Laplacian or sub-Laplacian on the general homogeneous groups to formulate a version similar to \eqref{eq:RellichE}. In fact, there may be no homogeneous hypoelliptic left-invariant differential operators at all: the existence of such an operator would imply that the group must be graded as was shown by Miller \cite{Miller} with further corrections by ter Elst and Robinson \cite{ER} (see also \cite[Proposition $4.1.3$]{FR} for a simple proof). Obviously, in the abelian case with the Euclidean norm, $\mR_2$ is nothing the radial Laplacian $\De_r = \pa_r^2 + \frac{n-1}{|x|}\pa_r$. Hence, \eqref{eq:RellichGRS} implies the following inequality in $\R^n$
\begin{equation}\label{eq:RellichEstrong}
\frac{n^2(n-4)^2}{16} \int_{\R^n} \frac{|f|^2}{|x|^4} dx \leq \int_{\R^n} |\De_r f|^2 dx,\quad n\geq 5.
\end{equation}
The inequality \eqref{eq:RellichEstrong} together with a recent result of Machihara, Ozawa and Wadade \cite{MOW2017} implies \eqref{eq:RellichE}. 

The higher order versions of \eqref{eq:RSHardy} in the case $p =2$ was also obtained in \cite{RS2017}. By iterating the weighted versions of \eqref{eq:RSHardy} with $p=2$, Ruzhansky and Suragan proved the following sharp weighted Hardy--Rellich type inequalities 
\begin{equation}\label{eq:RShigher}
\lt(\prod_{j=0}^{k-1} \lt|\frac{Q-2}2 -(\al + j)\rt|^2 \rt)\int_{\G} \frac{|f|^2}{|x|^{2k+2\al}} dx \leq \int_{\G} \frac{|\mR^k f|^2}{|x|^{2\al}} dx
\end{equation}
for any $Q\geq 3$, $k\geq 1$ and $\al \in \R$. The inequalities \eqref{eq:RellichGRS} and \eqref{eq:RShigher} are new even in the abelian case.

As mentioned above, the main aim of this paper is devoted to prove several inequalities of Hardy--Rellich type on the homogeneous groups. Our main results extended the inequality \eqref{eq:RellichGRS} in several ways. They give the generalizations of the Hardy--Rellich type inequalities on $\R^n$ mentioned above to the homogeneous groups with any homogeneous quasi-norm. First, we generalize it to any $p \in (1,Q/2)$ together with its weighted versions which is the content of Theorem \ref{identityRellich} and Theorem \ref{LpRellich} below. These obtained inequalities can be seen as the generalizations of the inequalities \eqref{eq:RellichE}, \eqref{eq:LpRellichE} and their weighted versions to the homogeneous groups with any homogeneous quasi-norm. Second, we extend it to the derivatives of higher order, i.e., we prove the Hardy--Rellich type inequality for $\mR_2^k$ and $\mR\mR_2^k$ for $k\geq 1$. This is the content of Theorem \ref{higherorderRellich} and Theorem \ref{eq:HRhigher}. These obtained inequalities also extend the inequalities \eqref{eq:higherRellichEeven}, \eqref{eq:higherRellichEodd}, \eqref{eq:Lphighereven} and \eqref{eq:Lphigherodd} to the homogeneous groups with any homogeneous quasi-norm. Thirst, we will prove the inequality \eqref{eq:RellichGRS} in the critical case, i.e., $p =Q/2$ and extend this critical inequality to the derivatives of higher order in Theorem \ref{criticalRellich} and Theorem \ref{criticalhigher} below respectively. Again, these are the extensions of the critical Rellich type inequalities \eqref{eq:ASLp}, \eqref{eq:ASL2}, \eqref{eq:ASchan} and \eqref{eq:ASle} to the homogeneous groups. Finally, we will generalize the inequality \eqref{eq:RonetwoE} together with its weighted versions and its $L^p-$versions to the homogeneous groups. These are the Rellich type inequalities obtained in Theorem \ref{onetwo} below for $p =2$ and Theorem \ref{onetwoLp} for any $1< p < Q/2$. By iterating these obtained inequalities, we will prove an extension of \eqref{eq:RonetwoE} to any order of derivatives in Theorem \ref{L2HRnew} and Theorem \ref{LpHRnew}. In application, we obtain, in Section $6$ below, several uncertainty principles on the homogeneous groups by using our obtained Hardy--Rellich type inequalities. It is worth to emphasize here that all the inequalities obtained in this paper are sharp and are derived from the corresponding equalities by dropping the nonnegative remainder terms. They are new even in the Euclidean space $\R^n$, and  give some new inequalities of Hardy--Rellich type both in the subcritical and critical cases in the setting of Euclidean space $\R^n$ (see Section $7$ below).


The rest of this paper is organized as follows. In Section $2$ we review briefly some basic properties of the homogeneous groups, fix the notation and recall the weighted $L^2-$Hardy inequalities on the homogeneous groups established by Ruzhansky and Suragan. In Section $3$ we prove some weighted $L^2-$Hardy--Rellich inequalities on the homogeneous groups. These inequalities include the weighted $L^2-$Rellich inequality (a weighted version \eqref{eq:RellichGRS}). We then use the weighted $L^2-$Hardy and weighted $L^2-$Rellich inequality to establish the weighted $L^2-$Hardy--Rellich type inequalities for higher order derivatives on the homogeneous groups. We also prove the Rellich type inequality that connects first to seconde order of derivatives and use it to derive the similar inequalities for any order of derivatives in this section. The $L^p-$versions of the Hardy--Rellich type iequalities obtained in Section $3$ (i.e., the weighted $L^p-$Hardy--Rellich type inequalities) will be established in Section $4$. The critical Hardy--Rellich type inequalities on the homogeneous groups are proved in Section $5$. In Section $6$, we obtain several new inequalities in the Euclidean space $\R^n$ by applying our results to this setting.

\section{Preliminaries}
In this section, we review briefly some basics of the analysis on homogeneous groups, some properties of the operator $\mR$ from \eqref{eq:mathR} and the weighted $L^2-$Hardy inequality on the homogeneous groups due to Ruzhansky and Suragan \cite{RS2017}. For the general background details on homogeneous groups, we refer the reader to the book of Folland and Stein \cite{FS1982} and the book of Fisher and Ruzhansky \cite{FR}.

We recall that a family of dilations of a Lie algebra $\mathfrak g$ is a family of linear mappings given by
\[
D_\lam = \exp(A \ln \lam) = \sum_{k=0}^\infty \frac1{k!} (A \ln \lam)^k,
\]
where $A$ is a diagonalisable linear operator on $\mathfrak g$ with positive egienvalues, and each $D_\lam$ is a morphism of the Lie algebra $\mathfrak g$, that is, a linear mapping from $\mathfrak g$ to itself which respects the Lie bracket:
\[
[D_\lam X, D_\lam Y] = D_\lam[X,Y],\quad \forall\, X, Y \in \mathfrak g,\, \lambda >0.
\]
A homogeneous groups is a simply connected Lie group whose Lie algebra is equipped with dilations. Homogeneous groups are necessarily nilpotent and hence the exponential mapping $\exp_{\mathbb G}: \mathfrak g \to \mathbb G$ is a global diffeomorphism. It induces a dilation structure on $\mathbb G$ which is still denoted by $D_\lam x$ or simply by $\lam x$ with $x \in \mathbb G$, i.e., 
\[
D_\lam x = \lam x: = \exp_{\mathbb G}(D_\lam (\exp_{\mathbb G}^{-1} x)).
\]
The homogeneous dimension of $\mathbb G$ is $\text{\rm Tr} A$ and is denoted by $Q$. 

Let $dx$ denote the Haar measure on $\mathbb G$ and let $|S|$ denote the corresponding volume of a measurable subset $S\subset \mathbb G$. Then we have
\[
|\lambda S| = \lam^Q |S|,\qquad \int_{\mathbb G}f(\lam x) dx = \lam^{-Q} \int_{\mathbb G} f(x) dx.
\]
Fix a basis $(X_1,\ldots,X_k)$ of $\mathfrak g$ such that $AX_k = \nu_k X_k$ for each $k$, so that $A$ has the form $A =\text{\rm diag}(\nu_1,\ldots,\nu_n)$ in this basis. Then each $X_k$ is homogeneous of degree $\nu_k$ and 
\[
Q =\nu_1 + \nu_2 + \cdots + \nu_n.
\]
The decomposition of $\exp_{\mathbb G}^{-1}(x)$ in the Lie algebra $\mathfrak g$ defines the vector
\[
e(x)= (e_1(x),\ldots,e_n(x))
\]
by the formula
\[
\exp_{\mathbb G}^{-1}(x) = e(x) \cdot \nabla = \sum_{i=1}^n e_i(x) X_i,
\]
where $\nabla =(X_1, \ldots,X_n)$. In the other word, we have
\[
x= \exp_{\mathbb G}\lt(\sum_{j=1}^n e_j(x) X_j\rt).
\]
By the homogeneity, we have
\[
r x = \exp_{\mathbb G} \lt(\sum_{j=1}^n r^{\nu_j} e_j(x) X_j\rt),
\]
that is, 
\[
e(rx) = (r^{\nu_1}e_1(x), \ldots, r^{\nu_n} e_n(x)).
\]
Consequently, we can calculate
\begin{align*}
\frac d{dr} (f(rx)) &= \frac d {dr} f\lt(\exp_{\mathbb G} \lt(\sum_{j=1}^n r^{\nu_j} e_j(x) X_j\rt)\rt)\\
&=\lt(\sum_{j=1}^n \nu_j r^{\nu_j-1} e_j(x) X_j\rt)f (rx).
\end{align*}
This implies the equality
\[
\frac d{dr}( f(rx)) = (\mathcal R f)(rx).
\]
In other words, the operator $\mR$ plays the role of the radial derivative on $\G$. Note that $\mR$ is positively homogeneous of order $-1$.

A homogeneous quasi-norm on a homogeneous group $\G$ is a continuous nonnegative function $\G\ni x \to |x|\in [0,\infty)$ satisfying the following conditions:
\begin{itemize}
\item $|x^{-1}| = |x|$ for any $x \in \G$,
\item $|\lam x| = \lam |x|$ for any $\lam >0$ and $x\in \G$,
\item $|x| =0$ if and only if $x =0$.
\end{itemize}

Let $|\cdot|$ be a quasi-norm on $\mathbb G$ and let $\mathfrak S$ denote the quasi-unit sphere with respect to $|\cdot|$, i.e.,
\[
\mathfrak S = \{x\in \mathbb G\, :\, |x| =1\}.
\]
It is well-known that there is a unique positive Borel measure $\sigma$ on $\mathfrak S$ such that for any function $f\in L^1(\G)$ we have
\begin{equation}\label{eq:polar}
\int_{\mathbb G} f(x) dx = \int_0^\infty \int_{\mathfrak S} f(ry) r^{Q-1} d\sigma(y) dr.
\end{equation}
We refer the reader the book of Folland and Stein \cite{FS1982} for the proof (see also \cite[Section $3.1.7$]{FR}).

In our analysis below, the following result whose proof can be found in \cite[Lemma $2.1$]{RS2017} plays an important role.
\begin{lemma}\label{homo}
Define the Euler's operator $\mathbb E$ on $\G$ by $\mathbb E = |x| \mR$. If $f: \mG\setminus\{0\} \to \R$ is continuously differentiable, then
\[
\mathbb Ef = \nu f\, \quad\text{\rm if and onlu if}\quad\, f(\lam x) = \lam^\nu f(x),\quad \forall\, \lam >0, x\not=0,
\]
i.e., $f$ is positively homogeneous of order $\nu$.
\end{lemma} 

We conclude this section by recalling a weighted $L^2-$Hardy inequality on the homogeneous group due to Ruzhansky and Suragan (see \cite[Theorem $4.1$ and Corollary $4.2$]{RS2017}).
\begin{theorem}\label{identityHardy}
Let $\mathbb G$ be a homogeneous group of homogeneous dimension $Q\geq 3$ and let $|\cdot|$ be any homogeneous quasi-norm on $\mathbb G$. Then for any complex-valued function $f \in C^\infty_0(\mathbb G\setminus\{0\})$ we have
\begin{equation}\label{eq:identityHardy}
\int_{\mathbb G} \frac{|\mathcal R f|^2}{|x|^{2\al}} dx = \lt(\frac{Q-2-2\al}2\rt)^2 \int_{\mathbb G} \frac{|f|^2}{|x|^{2\al +2}} dx + \int_{\mathbb G} \lt|\frac{\mathcal R f}{|x|^{\al}} + \frac{Q-2-2\al}{2|x|^{\al +1}} f\rt|^2 dx
\end{equation}
for any $\al \in \R$. As a corollary, we have the following weighted $L^2-$Hardy inequality on $\G$,
\begin{equation}\label{eq:weightedHardy}
\lt(\frac{Q-2-2\al}2\rt)^2 \int_{\mathbb G} \frac{|f|^2}{|x|^{2\al +2}} dx\leq \int_{\mathbb G} \frac{|\mathcal R f|^2}{|x|^{2\al}} dx.
\end{equation}
The constant in \eqref{eq:weightedHardy} is sharp and it is attained if and only if $f =0$.
\end{theorem}
\section{Weighted $L^2-$Hardy--Rellich type inequalities}
Throughout this section, we use the notation 
\[
c_\al = \frac{(Q+2\al)(Q -4-2\al)}4
\]
for any $\al \in \R$, here $Q$ is the homogeneous dimension of a homogeneous group $\mG$. We start this section by proving an interesting equality which implies the weighted $L^2-$Rellich inequality \eqref{eq:RellichGRS} on $\G$. This can be seen as a weighted version of Theorem $5.1$ in \cite{RS2017}. Recall that 
\[
\mathcal R_2 f = \mR^2f + \frac{Q-1}{|x|} \mR f.
\]
We then have following results.

\begin{theorem}\label{identityRellich}
Let $\mathbb G$ be a homogeneous group of homogeneous dimension $Q\geq 5$ and let $|\cdot|$ be any homogeneous quasi-norm on $\mathbb G$. Then for any complex-valued function $f \in C^\infty_0(\mathbb G\setminus\{0\})$ we have
\begin{align}\label{eq:identityRellich}
\int_{\mathbb G} \frac{|\mathcal R_2 f|^2}{|x|^{2\al}} dx&= c_\al^2 \int_{\mathbb G} \frac{|f|^2}{|x|^{2\al +4}}dx + \int_{\mathbb G} \lt|\frac{\mathcal R_2 f}{|x|^{\al}} + c_\al \frac f{|x|^{\al +2}}\rt|^2 dx \notag\\
&\qquad\qquad\qquad  + 2 c_\al \int_{\mathbb G} \lt|\frac{\mathcal R f}{|x|^{1+\al}} + \frac{Q-4-2\al}{2 |x|^{2+ \al}} f\rt|^2 dx,
\end{align}
for any $\al \in \R$. The constant $c_\al^2$ before $\int_{\mathbb G} |x|^{-4 -2\al} |f|^2 dx$ in the right hand side of \eqref{eq:identityRellich} is sharp. As a consequence, we obtain the following weighted Rellich inequality for any complex-valued function $f \in C^\infty_0(\mathbb G\setminus\{0\})$
\begin{equation}\label{eq:weightedRellich}
\lt(\frac{(Q+2\al)(Q-4-2\al)}4\rt)^2\int_{\mathbb G} \frac{|f|^2}{|x|^{2\al +4}} dx \leq \int_{\mathbb G} \frac{\lt|\mathcal R_2 f \rt|^2}{|x|^{2\al}}  dx,
\end{equation}
for any $\al \in (-Q/2,(Q-4)/2)$. The constant in \eqref{eq:weightedRellich} is sharp and it is attained if and only if $f =0$.
\end{theorem}
If $\al =0$, Theorem \ref{identityRellich} recovers Theorem $5.1$ and Corollary $5.2$ in \cite{RS2017}. The proof of Theorem \ref{identityRellich} follows the lines in the proof of Theorem $5.1$ in \cite{RS2017} (see also the proof of Theorem $1.1$ in \cite{MOW2017} for the Euclidean space $\R^n$).


\begin{proof}
Using the polar coordinate \eqref{eq:polar}, we have
\begin{align}\label{eq:R1}
\int_{\mathbb G} &\frac{|f|^2}{|x|^{4+2\al}} dx\notag\\
&= \int_0^\infty r^{Q-5-2\al} \int_{\mathfrak S} |f(r y)|^2 d\sigma(y) dr\notag\\
&=\frac{1}{Q-4-2\al} \int_0^\infty (r^{Q-4-2\al})'\int_{\mathfrak S} |f(r y)|^2 d\sigma(y) dr\notag\\
&= -\frac2{Q-4-2\al} \Re\int_0^\infty r^{Q-4-2\al} \int_{\mathfrak S} f(r y) \overline{\mathcal R f(ry)} d\si(y) dr\notag\\
&= -\frac2{(Q-4-2\al)(Q-3-2\al)} \Re\int_0^\infty (r^{Q-3-2\al})' \int_{\mathfrak S} f(r y) \overline{\mathcal R f(ry)} d\si(y) dr\notag\\
&=\frac2{(Q-4-2\al)(Q-3-2\al)} \Re\int_0^\infty r^{Q-3-2\al} \int_{\mathfrak S} \lt(|\mathcal R f(ry)|^2 + f(r y) \overline{\mathcal R^2 f(ry)} \rt)d\si(y)dr\notag\\
&=\frac2{(Q-4-2\al)(Q-3-2\al)} \Re \int_{\mathbb G} \lt(\frac{|\mathcal R f|^2}{|x|^{2+2\al}} + \frac{f \overline{\mathcal R^2 f}}{|x|^{2+ 2\al}}\rt) dx,
\end{align}
here $\Re z$ denotes the real part of a complex number $z\in \C$. From Theorem \ref{identityHardy}, we have
\begin{equation}\label{eq:R2}
\int_{\mathbb G} \frac{|\mathcal R f|^2}{|x|^{2+2\al}} dx = \frac{(Q-4-2\al)^2}4 \int_{\mathbb G} \frac{|f|^2}{|x|^{4+2\al}} dx + \int_{\mathbb G} \lt|\frac{\mathcal R f}{|x|^{1+\al}} + \frac{Q-4-2\al}{2|x|^{2+\al}} f\rt|^2 dx
\end{equation}
Using integration by parts, we have
\begin{align}\label{eq:R3}
\Re\int_{\mathbb G} \frac{f \overline{\mathcal R^2 f}}{|x|^{2+ 2\al}} dx &=\Re \int_{\mathbb G} \frac{f \, \overline{\mathcal R_2 f}}{|x|^{2+2\al}}  dx -(Q-1) \Re \int_{\mathbb G}\frac{f \overline{\mathcal R f}}{|x|^{3+ 2\al}} dx\notag\\
&= \Re \int_{\mathbb G} \frac{f \, \overline{\mathcal R_2 f}}{|x|^{2+2\al}}  dx + \frac{(Q-1)(Q-4-2\al)}2\int_{\mathbb G} \frac{|f|^2}{|x|^{4+ 2\al}} dx.
\end{align}
Plugging \eqref{eq:R2} and \eqref{eq:R3} into \eqref{eq:R1} implies
\begin{align*}
\int_{\mathbb G} \frac{|f|^2}{|x|^{4+ 2\al}} dx& = -\frac1{c_\al} \Re \int_{\mathbb G} \frac{f \, \overline{\mathcal R_2 f}}{|x|^{2+2\al}}  dx -\frac1{c_\al} \int_{\mathbb G} \lt|\frac{\mathcal R f}{|x|^{1+\al}} + \frac{Q-4-2\al}{2|x|^{2+\al}} f\rt|^2 dx\\
&=-\frac1{2 c_\al^2}\int_{\mathbb G} \int_{\mathbb G} \lt|\frac{\mathcal R_2 f}{|x|^{\al}} + c_\al \frac f{|x|^{\al +2}}\rt|^2 dx + \frac12 \int_{\mathbb G} \frac{|f|^2}{|x|^{4+ 2\al}} dx\\
&\quad + \frac1{2 c_\al^2}\int_{\mathbb G} \frac{|\mathcal R_2 f|^2}{|x|^{2\al}}  dx -\frac1{c_a} \int_{\mathbb G} \lt|\frac{\mathcal R f}{|x|^{1+\al}} + \frac{Q-4-2\al}{2|x|^{2+\al}} f\rt|^2 dx
\end{align*}
which implies our desired result \eqref{eq:identityRellich}.

It remains to verify the sharpness of $c_\al^2$ in \eqref{eq:identityRellich}. To do this, we will use the approximations of the function $r^{-(Q-4-2\al)/2}$ as follows. Let $\phi \in C_0^\infty(\R)$ such that $\phi =1$ on $(-1,1)$ and $\phi =0$ on $\R \setminus (-2,2)$. For any $\ep >0$, define $f_\ep(r) = (1-\phi(r/\ep)) r^{-\frac{Q-4-2\al}2} \phi(\ep r)$. Differentiating $f_\ep$ we get
\begin{align*}
\mathcal R f_\ep(r) &= -\frac1\ep \phi'(\ep^{-1} r) r^{-\frac{Q-4-2\al}2} \phi(\ep r) -\frac{Q-4-2\al}2 (1-\phi(\ep^{-1}r)) r^{-\frac{Q-2-2\al}2} \phi(\ep r)\\
&\qquad  + \ep (1-\phi(r/\ep)) r^{-\frac{Q-4-2\al}2} \phi'(\ep r),
\end{align*}
and
\begin{align*}
\mathcal R^2 f_\ep(r)& = -\frac1{\ep^2} \phi''(\ep^{-1} r) r^{-\frac{Q-4-2\al}2} \phi(\ep r)+ \ep^2 (1-\phi(\ep^{-1}r)) r^{-\frac{Q-4-2\al}2} \phi''(\ep r)\\
&\quad +\frac{Q-4-2\al}\ep \phi'(\ep^{-1}r)) r^{-\frac{Q-2-2\al}2} \phi(\ep r)- \ep (Q-4-2\al)(1-\phi(\ep^{-1}r)) r^{-\frac{Q-2-2\al}2} \phi'(\ep r)\\
&\quad + \frac{(Q-4-2\al)(Q-2-2\al)}4 (1-\phi(\ep^{-1}r)) r^{-\frac{Q-2\al}2} \phi(\ep r),
\end{align*}
here we use the supports of $\phi'(r/\ep)$ and $\phi'(\ep r)$ are disjoint for $\ep >0$ small enough. Thus, we obtain
\begin{align*}
\mathcal R_2 f_\ep(r)&= -\frac1{\ep^2} \phi''(\ep^{-1} r) r^{-\frac{Q-4-2\al}2} \phi(\ep r)+ \ep^2 (1-\phi(\ep^{-1}r)) r^{-\frac{Q-4-2\al}2} \phi''(\ep r)\\
&\quad -\frac{3+2\al}\ep \phi'(\ep^{-1}r)) r^{-\frac{Q-2-2\al}2} \phi(\ep r)+ \ep (3+2\al)(1-\phi(\ep^{-1}r)) r^{-\frac{Q-2-2\al}2} \phi'(\ep r)\\
&\quad -c_\al(1-\phi(\ep^{-1}r)) r^{-\frac{Q-2\al}2} \phi(\ep r).
\end{align*}
This implies
\[
\int_{\mathbb G} \frac{\lt|\mathcal R_2 f_\ep\rt|^2}{|x|^{2\al}}  dx = c_\al^2 (-\ln \ep) \sigma(\mathfrak S) + O(1).
\]
We can easily check that
\[
\int_{\mathbb G} \frac{|f_\ep|^2}{|x|^{4+ 2\al}} dx = (-\ln \ep) \sigma(\mathfrak S) + O(1),
\]
\[
\int_{\mathbb G} \lt|\frac{\mathcal R^2 f_\ep}{|x|^{\al}} + \frac{Q-1}{|x|^{1+\al}}\mathcal Rf_\ep + c_\al \frac {f_\ep} {|x|^{\al +2}}\rt|^2 dx = O(1),
\]
and
\[
\int_{\mathbb G} \lt|\frac{1}{|x|^{1+\al}}\mathcal R f_\ep + \frac{Q-4-2\al}{2 |x|^{2+ \al}} f_\ep\rt|^2 dx = O(1).
\]
These computations show that the constant $c_\al^2$ is sharp in \eqref{eq:identityRellich}.

The inequality \eqref{eq:weightedRellich} is trivial since $c_\al >0$ for $\al \in (-Q/2, (Q-4)/2)$. We next verify the sharpness of constant. For $\ep >0$, consider the function $f_\ep$ as above. We then have 
\[
\lim_{\ep\to 0} \frac{\int_{\mathbb G} \frac1{|x|^{2\al}} \lt|\mathcal R^2 f_\ep + \frac{Q-1}{|x|}\mathcal Rf_\ep \rt|^2 dx}{\int_{\mathbb G} \frac{|f_\ep|^2}{|x|^{4+ 2\al}} dx} = c_\al^2,
\]
which implies the sharpness of $c_\al^2$.

Suppose that there exists equality in \eqref{eq:weightedRellich} for some function $f$. From Theorem \ref{identityRellich}, we must have
\[
\mathcal R f + \frac{Q-4-2\al}{2|x|} f =0\quad \Leftrightarrow \quad\mathbb Ef = -\frac{Q-4-2\al}{2} f.
\]
Lemm \ref{homo} implies that $f$ is positively homogeneous of order $-(Q-4-2\al)/2$, i.e., there exists function $h: \mathfrak S \to \mathbb C$ such that $f(x) = |x|^{-(Q-4-2\al)/2} h(x/|x|)$. Since $f(x)/|x|^{2+\al}$ is in $L^2(\mathbb G)$, we then have $h =0$ on $\mathfrak S$ or equivalently $f =0$ on $\mG$.
\end{proof}

We continue by extending Theorem \ref{identityHardy} and Theorem \ref{identityRellich} to derivatives of higher order. This will be done in the following theorem.
 
\begin{theorem}\label{higherorderRellich}
Let $\mathbb G$ be a homogeneous group of homogeneous dimension $Q$, and let $|\cdot|$ be a quasi-norm on $\mathbb G$ and $k$ be a positive integer and $\al \in \R$. For any complex-valued function $f\in C_0^\infty(\mathbb G\setminus\{0\})$ we have
\begin{align}\label{eq:even}
&\lt(\prod_{i=0}^{k-1} c_{2i+ \al}\rt)^2\int_{\mathbb G} \frac{|f|^2}{|x|^{4k+2\al}} dx \notag\\
&=\int_{\mathbb G} \frac{|\mathcal R_2^k f|^2}{|x|^{2\al}} dx -\int_{\mathbb G} \frac1{|x|^{2\al}} \lt|\mathcal R_2^k f + c_\al \frac{\mathcal R_2^{k-1}f}{|x|^2}\rt|^2 dx \notag\\
&\quad -\sum_{j=1}^{k-1}\lt(\prod_{i=0}^{j-1} c_{2i+ \al}\rt)^2 \int_{\mathbb G} \frac1{|x|^{4j+ 2\al}} \lt|\mathcal R_2^{k-j}f + c_{2j+\al} \frac{\mathcal R_2^{k-j-1}f}{|x|^2}\rt|^2 dx\notag\\
&\quad -2c_\al \int_{\mathbb G} \frac1{|x|^{2+2\al}} \lt|\mathcal R (\mathcal R_2^{k-1}f) + \frac{Q-4-2\al}{2|x|} \mathcal R_2^{k-1} f\rt|^2 dx\notag\\
&\quad -2 \sum_{j=1}^{k-1} \lt(\prod_{i=0}^{j-1} c_{2i+ \al}\rt)^2 c_{2j+\al} \int_{\mathbb G} \frac1{|x|^{2+2\al + 4j}} \lt|\mathcal R (\mathcal R_2^{k-j-1}f)+ \frac{Q-4-2\al-4j}{2|x|} \mathcal R_2^{k-j-1} f\rt|^2 dx,
\end{align}
if $Q \geq 4k+1$, and
\begin{align}\label{eq:odd}
&\lt(\frac{Q-2-2\al}2 \prod_{i=0}^{k-1} c_{2i+1+ \al}\rt)^2\int_{\mathbb G}\frac{|f|^2}{|x|^{4k+2+ 2\al}} dx\notag\\
&=\int_{\mathbb G} \frac{|\mathcal R(\mathcal R_2^k f)|^2}{|x|^{2\al}} dx -\int_{\mathbb G} \frac1{|x|^{2\al}} \lt|\mathcal R(\mathcal R_2^k f) + \frac{Q-2-2\al}{2|x|} \mathcal R_2^{k}f \rt|^2 dx \notag\\
&\quad -\frac{(Q-2-2\al)^2}4 \int_{\mathbb G} \frac1{|x|^{2+2\al}} \lt|\mathcal R_2^k f + c_{1+\al} \frac{\mathcal R_2^{k-1}f}{|x|^2}\rt|^2 dx \notag\\
&\quad -\frac{(Q-2-2\al)^2}4\sum_{j=1}^{k-1} \lt(\prod_{i=0}^{j-1} c_{2i+ \al}\rt)^2\int_{\mathbb G} \frac1{|x|^{2+2\al+ 4j}} \lt|\mathcal R_2^{k-j} f + c_{2j+ 1+\al} \frac{\mathcal R_2^{k-j-1}f}{|x|^2}\rt|^2 dx \notag\\
&\quad -\frac{(Q-2-2\al)^2}2 c_{1+\al}\int_{\mathbb G} \frac1{|x|^{4+2\al}}\lt|\mathcal R(\mathcal R_2^{k-1} f) + \frac{Q-6-2\al}{2|x|} \mathcal R_2^{k-1} f\rt|^2 dx\notag\\
&\quad -\frac{(Q-2-2\al)^2}2\sum_{j=1}^{k-1}\lt(\prod_{i=0}^{j-1} c_{2i+ 1+\al}\rt)^2 c_{2j+1+\al} \times\notag\\
&\qquad\qquad\qquad\qquad \times \int_{\mathbb G} \frac1{|x|^{4+2\al+4j}}\lt|\mathcal R(\mathcal R_2^{k-j-1} f) + \frac{Q-6-2\al-4j}{2|x|} \mathcal R_2^{k-j-1} f\rt|^2 dx
\end{align}
if $Q \geq 4k + 3$.

As consequence, we following weighted Hardy--Rellich type inequalities for any complex-valued function $f\in C_0^\infty(\mathbb G\setminus\{0\})$
\begin{equation}\label{eq:HReven}
\lt(\prod_{i=0}^{k-1} c_{2i+ \al}\rt)^2\int_{\mathbb G} \frac{|f|^2}{|x|^{4k+2\al}} dx \leq \int_{\mathbb G} \frac{|\mathcal R_2^k f|^2}{|x|^{2\al}} dx
\end{equation}
if $Q \geq 4k+1$ and $\al \in (-Q/2,(Q-4k)/2)$, and
\begin{equation}\label{eq:HRodd}
\lt(\frac{Q-2-2\al}2 \prod_{i=0}^{k-1} c_{2i+1+ \al}\rt)^2\int_{\mathbb G}\frac{|f|^2}{|x|^{4k+2+2\al}} dx \leq \int_{\mathbb G} \frac{|\mathcal R(\mathcal R_2^k f)|^2}{|x|^{2\al}} dx
\end{equation}
if $Q \geq 4k+3$, $k\geq 1$ and $\al \in (-(Q+2)/2, (Q-4k-2)/2)$. Moreover, these inequalities are sharp and the equality holds if and only if $f =0$.
\end{theorem}
\begin{proof}
We first prove \eqref{eq:even}. If $k=1$, it is exactly \eqref{eq:identityRellich} from Theorem \ref{identityRellich}. If $k >1$, it follows by induction argument using consecutively Theorem \ref{identityRellich}.

We next prove \eqref{eq:odd}. Using Theorem \ref{identityHardy}, we get
\begin{align}\label{eq:R1new}
&\int_{\mathbb G} \frac{|\mathcal R(\mathcal R_2^k f)|^2}{|x|^{2\al}} dx\notag\\
&= \frac{(Q-2-2\al)^2}4 \int_{\mathbb G} \frac{|\mathcal R_2^k f|^2}{|x|^{2 + 2\al}} dx + \int_{\mathbb G} \frac1{|x|^{2\al}} \lt|\mathcal R(\mathcal R_2^k f) + \frac{Q-2-2\al}{2|x|} \mathcal R_2^k f\rt|^2 dx.
\end{align}
Thus, \eqref{eq:odd} follows from \eqref{eq:even} and \eqref{eq:R1new}.

The inequalities \eqref{eq:HReven} and \eqref{eq:HRodd} are trivial since $c_{2i+\al} >0, c_{2i+1+\al} >0$ for any $i=0,\ldots,k-1$ corresponding to each case. To check the sharpness of constant, we use the arguments as in the proof of Theorem \ref{identityRellich} by approximating the function $r^{-(Q-4k-2\al)/2}$. The straightforward and tedious computations show that 
\[
\int_{\mathbb G} \frac{|f_\ep|^2}{|x|^{4k+ 2\al}} dx = (-\ln \ep) \sigma(\mathfrak S) + O(1),
\]
and
\[
\int_{\mathbb G} \frac{|\mathcal R_2^k f|^2}{|x|^{2\al}} dx = \lt(\prod_{i=0}^{k-1} c_{2i+ \al}\rt)^2 (-\ln \ep) \sigma(\mathfrak S) + O(1).
\]
This proves the sharpness of \eqref{eq:HReven}. For \eqref{eq:HRodd}, we use the approximations of the function $r^{-(Q-2-4k-2\al)/2}$ and make the same computations.

If equality holds in \eqref{eq:HReven} for a function $f$. From \eqref{eq:even}, we see that
\[
\mathcal R f + \frac{Q-2\al -4k}{2|x|} f =0\quad \Longleftrightarrow \quad \mathbb Ef = -\frac{Q-2\al -4k}{2} f.
\]
By Lemma \ref{homo}, $f$ is positive homogeneous of order $-(Q-2\al -4k)/2$ which forces $f =0$ since $f/|x|^{\al + 2k} \in L^2(\mathbb G)$. The same argument also works for the equality in \eqref{eq:HRodd}.
\end{proof}

The rest of this section is to establish the generalizations of \eqref{eq:RonetwoE} to the homogeneous groups and its higher order versions. We first extend \eqref{eq:RonetwoE} to a weighted version on the homogeneous groups as follows.

\begin{theorem}\label{onetwo}
Let $\G$ be a homogeneous group of homogeneous dimension $Q\geq 5$. Let $|\cdot|$ be any homogeneous quasi-norm on $\mathbb G$. Then for any complex-valued function $f\in C_0^\infty(\mathbb G\setminus\{0\})$, we have
\begin{equation}\label{eq:identityonetwo}
\int_{\mathbb G} \frac{|\mathcal R_2 f|^2}{|x|^{2\al}} dx = \frac{(Q+2\al)^2}4 \int_{\mathbb G} \frac{|\mathcal R f|^2}{|x|^{2+2\al}} dx + \int_{\mathbb G} \frac1{|x|^{2\al}} \lt|\mathcal R^2 f + \frac{Q-2-2\al}{2|x|} \mathcal R f\rt|^2 dx,
\end{equation}
for any $\al \in \R$. As consequence, we get the following Rellich type inequality
\begin{equation}\label{eq:newRellich}
\frac{(Q+2\al)^2}4 \int_{\mathbb G} \frac{|\mathcal R f|^2}{|x|^{2+2\al}} dx \leq \int_{\mathbb G} \frac{|\mathcal R_2 f|^2}{|x|^{2\al}} dx
\end{equation}
for any complex-valued function $f\in C_0^\infty(\mathbb G\setminus\{0\})$ and $\al \in \R$. The inequality \eqref{eq:newRellich} is sharp and equality hold if and only if $f =0$.
\end{theorem}
\begin{proof}
Expanding the square of $\mathcal R_2 f$, we get
\[
\int_{\mathbb G} \frac{|\mathcal R_2 f|^2}{|x|^{2\al}} dx = \int_{\mathbb G} \frac{|\mathcal R(\mathcal R f)|^2}{|x|^{2\al}} dx + 2(Q-1) \Re \int_{\mathbb G}\frac{\mathcal R(\mathcal R f)\, \overline{\mathcal R f}}{|x|^{2\al +1}} dx + (Q-1)^2 \int_{\mathbb G} \frac{|\mathcal R f|^2}{|x|^{2+2\al}} dx.
\]
Using integration by parts, we have
\[
2\Re \int_{\mathbb G}\frac{\mathcal R(\mathcal R f)\, \overline{\mathcal R f}}{|x|^{2\al +1}} dx = -(Q-2-2\al) \int_{\mathbb G} \frac{|\mathcal R f|^2}{|x|^{2+2\al}} dx.
\]
By Theorem \ref{identityHardy}, we have
\[
\int_{\mathbb G} \frac{|\mathcal R(\mathcal R f)|^2}{|x|^{2\al}} dx = \frac{(Q-2-2\al)^2}4  \int_{\mathbb G} \frac{|\mathcal R f|^2}{|x|^{2+2\al}} dx + \int_{\mathbb G} \frac1{|x|^{2\al}} \lt|\mathcal R^2 f + \frac{Q-2-2\al}{2|x|} \mathcal R f\rt|^2 dx.
\]
Gathering these equalities together, we obtain \eqref{eq:identityonetwo}.

The inequality \eqref{eq:newRellich} is an immediate consequence of \eqref{eq:identityonetwo}. The sharpness of \eqref{eq:newRellich} is verified by using approximations of the function $r^{-(Q-4-2\al)/2}$. If equality in \eqref{eq:newRellich} holds for some function $f$, then we must have
\[
\mathcal R^2 f + \frac{Q-2-2\al}{2|x|} \mathcal R f =0 \quad \Longleftrightarrow \quad \mathbb E(\mathcal Rf) = -\frac{Q-2-2\al}{2} \mathcal R f
\]
 which implies that $\mathcal R f$ is positively homogeneous of degree $-(Q-2-2\al)/2$ by Lemma \ref{homo}. This will implies $\mathcal R f =0$ since $\mathcal R f/|x|^{1+\al}$ is in $L^2(\mathbb G)$. Hence, so is $f$.
\end{proof}

Combining Theorem \ref{higherorderRellich} and Theorem \ref{onetwo}, we obtain some new weighted Rellich type inequalities. 
\begin{theorem}\label{L2HRnew}
Let $G$ be a homogeneous group of homogeneous dimension $Q$. Let $|\cdot|$ be any homogeneous quasi-norm on $\mathbb G$. Let $k,l$ be nonnegative integers such that $Q \geq 4k+1$ and $k \geq l+1$. Then for any complex-valued function $f\in C_0^\infty(\mathbb G\setminus\{0\})$ and $\al \in \R$, we have
\begin{align}\label{eq:L2derivativeeven}
\int_{\mathbb G}& \frac{|\mathcal R_2^k f|^2}{|x|^{2\al}} dx\notag\\
&\quad=\frac4{(Q-2\al)^2}\lt(\prod_{i=0}^{k-l-1} \frac{Q^2-4(2i+\al)^2}4\rt)^2\int_{\mathbb G} \frac{|\mathcal R \mathcal R_2^l f|^2}{|x|^{4(k-l)-2+ 2\al}} dx \notag\\
&\quad\quad + \lt(\prod_{i=0}^{k-l-2} c_{2i+ \al}\rt)^2\int_{\mathbb G} \frac{\lt|\mathcal R^2 \mathcal R_2^l f + \frac{Q+2-4(k-l) -2\al}{2|x|} \mathcal R \mathcal R_2^l f\rt|^2}{|x|^{4(k-l-1)+2\al}} dx\notag\\
&\quad\quad  +\int_{\mathbb G} \frac{\lt|\mathcal R_2^k f + c_\al \frac{\mathcal R_2^{k-1}f}{|x|^2}\rt|^2}{|x|^{2\al}}  dx  +\sum_{j=1}^{k-l-2}\lt(\prod_{i=0}^{j-1} c_{2i+ \al}\rt)^2 \int_{\mathbb G} \frac{ \lt|\mathcal R_2^{k-j}f + c_{2j+\al} \frac{\mathcal R_2^{k-j-1}f}{|x|^2}\rt|^2}{|x|^{4j+ 2\al}} dx\notag\\
&\quad\quad +2c_\al \int_{\mathbb G} \frac{ \lt|\mathcal R (\mathcal R_2^{k-1}f) + \frac{Q-4-2\al}{2|x|} \mathcal R_2^{k-1} f\rt|^2}{|x|^{2+2\al}} dx\notag\\
&\quad\quad +2 \sum_{j=1}^{k-l-2} \lt(\prod_{i=0}^{j-1} c_{2i+ \al}\rt)^2 c_{2j+\al} \int_{\mathbb G} \frac{\lt|\mathcal R (\mathcal R_2^{k-j-1}f)+ \frac{Q-4-2\al-4j}{2|x|} \mathcal R_2^{k-j-1} f\rt|^2}{|x|^{2+2\al + 4j}}  dx,
\end{align}
and
\begin{align}\label{eq:L2derivativeodd}
\int_{\mathbb G} &\frac{|\mathcal R \mathcal R_2^k f|^2}{|x|^{2\al}} dx\notag\\
&=\lt(\prod_{i=0}^{k-l-1}\frac{Q^2 -4(1+\al +2i)^2}4\rt)^2\int_{\mathbb G} \frac{|\mathcal R \mathcal R_2^l f|^2}{|x|^{4(k-l)+ 2\al}} dx \notag\\
&\quad + \frac{(Q-2-2\al)^2}4\lt(\prod_{i=0}^{k-l-2} c_{2i+1+\al}\rt)^2\int_{\mathbb G} \frac{\lt|\mathcal R^2 \mathcal R_2^l f + \frac{Q+2-4(k-l) -2\al}{2|x|} \mathcal R \mathcal R_2^l f\rt|^2}{|x|^{4(k-l-1)+2+2\al}} dx\notag\\
&\quad +\frac{(Q-2-2\al)^2}4\int_{\mathbb G} \frac{\lt|\mathcal R_2^k f + c_{1+\al} \frac{\mathcal R_2^{k-1}f}{|x|^2}\rt|^2}{|x|^{2+2\al}}  dx \notag\\
&\quad +\frac{(Q-2-2\al)^2}4\sum_{j=1}^{k-l-2}\lt(\prod_{i=0}^{j-1} c_{2i+ 1+\al}\rt)^2 \int_{\mathbb G} \frac{ \lt|\mathcal R_2^{k-j}f + c_{2j+1+\al} \frac{\mathcal R_2^{k-j-1}f}{|x|^2}\rt|^2}{|x|^{4j+2+ 2\al}} dx\notag\\
&\quad +2c_{1+\al}\frac{(Q-2-2\al)^2}4 \int_{\mathbb G} \frac{ \lt|\mathcal R (\mathcal R_2^{k-1}f) + \frac{Q-6-2\al}{2|x|} \mathcal R_2^{k-1} f\rt|^2}{|x|^{4+2\al}} dx\notag\\
&\quad+2 \frac{(Q-2-2\al)^2}4\sum_{j=1}^{k-l-2} \lt(\prod_{i=0}^{j-1} c_{2i+1+ \al}\rt)^2 c_{2j+1+\al}\times\notag\\
&\qquad\qquad \qquad\qquad \qquad\qquad\qquad\times \int_{\mathbb G} \frac{\lt|\mathcal R (\mathcal R_2^{k-j-1}f)+ \frac{Q-6-2\al-4j}{2|x|} \mathcal R_2^{k-j-1} f\rt|^2}{|x|^{4+2\al + 4j}}  dx\notag\\
&\quad + \int_{\mathbb G} \frac{ \lt|\mathcal R \mathcal R_2^k f+ \frac{Q-2 -2\al}2 \frac{\mathcal R_2^k f}{|x|}\rt|^2}{|x|^{2\al}} dx.
\end{align}
As a consequence, the following weighted Rellich type inequalities holds for any complex-valued function $f \in C_0^\infty(\mathbb G\setminus\{0\})$
\begin{equation}\label{eq:L2HRneweven}
\frac4{(Q-2\al)^2}\lt(\prod_{i=0}^{k-l-1} \frac{Q^2-4(2i+\al)^2}4\rt)^2\int_{\mathbb G} \frac{|\mathcal R\mathcal R_2^l f|^2}{|x|^{4(k-l) -2+ 2\al}} dx \leq \int_{\mathbb G} \frac{|\mathcal R_2^k f|^2}{|x|^{2\al}} dx,
\end{equation}
for any $\al \in (-Q/2, (Q-4(k-l-1))/2)$ if $k\geq l+2$ and $\al \in \R$ if $k=l+1$, and
\begin{equation}\label{eq:L2HRnewodd}
\lt(\prod_{i=0}^{k-l-1}\frac{Q^2 -4(1+\al +2i)^2}4\rt)^2\int_{\mathbb G} \frac{|\mathcal R\mathcal R_2^l f|^2}{|x|^{4(k-l)+ 2\al}} dx \leq \int_{\mathbb G} \frac{|\mathcal R \mathcal R_2^k f|^2}{|x|^{2\al}} dx,
\end{equation}
for any $\al \in (-(Q+2)/2, (Q-4(k-l)+2)/2)$ if $k\geq l+2$ and $\al \in \R$ if $k=l+1$. Moreover, these inequalities \eqref{eq:L2HRneweven} and \eqref{eq:L2HRnewodd} are sharp and equality holds if and only if $f=0$.
\end{theorem}
\begin{proof}
We first prove \eqref{eq:L2derivativeeven}. Denote $g =\mathcal R^{l+1}_2 f$ and apply \eqref{eq:even} for $\mathcal R_2^{k-l-1}g$ and then apply Theorem \ref{onetwo} for $\mathcal R_2 \mathcal R_2^l f$, we get the desired equality \eqref{eq:L2derivativeeven}. Here we use the equality
\[
\lt(\frac{Q+4(k-l-1)+2\al}2\rt)^2\lt(\prod_{i=0}^{k-l-2} c_{2i+ \al}\rt)^2 = \frac4{(Q-2\al)^2}\lt(\prod_{i=0}^{k-l-1} \frac{Q^2-4(2i+\al)^2}4\rt)^2.
\]
To prove \eqref{eq:L2derivativeodd}, we first apply \eqref{eq:identityHardy} to $\mathcal R\mathcal R_2^k f$ and then use \eqref{eq:L2derivativeeven} for $\mathcal R_2^k f$ with weights $|x|^{2(1+\al)}$ and the equality
\begin{multline*}
\frac{(Q-2-2\al)^2}4\frac{(Q+4(k-l-1)+2+2\al)^2}4\lt(\prod_{i=0}^{k-l-2} c_{2i+ 1+\al}\rt)^2 \\
=\lt(\prod_{i=0}^{k-l-1}\frac{Q^2 -4(1+\al +2i)^2}4\rt)^2.
\end{multline*}

The inequalities \eqref{eq:L2HRneweven} and \eqref{eq:L2HRnewodd} are implied from \eqref{eq:L2derivativeeven} and \eqref{eq:L2derivativeodd} respectively by dropping the nonnegative remainder terms. The sharpness of \eqref{eq:L2HRneweven} and \eqref{eq:L2HRnewodd} is verified by using the approximations of the function $r^{-(Q-4k-2\al)/2}$ and the function $r^{-(Q-4k-2-2\al)/2}$ respectively. 

Suppose that equality holds in \eqref{eq:L2HRneweven} for a function $f$. It follows from \eqref{eq:L2derivativeeven} that 
\[
\mathcal R (\mathcal R_2^{k-1}f) + \frac{Q-4-2\al}{2|x|} \mathcal R_2^{k-1} f =0 \quad \Longleftrightarrow \quad \mathbb E(\mathcal R_2^{k-1}f) = -\frac{Q-4-2\al}{2} \mathcal R_2^{k-1} f.
\]
By Lemma \ref{homo}, $\mathcal R_2^{k-1} f$ is positively homogeneous of order $-(Q-4-2\al)/2$. Since $\mathcal R_2^{k-1} f/|x|^{2+ \al} \in L^2(\mathbb G)$, we then must have $\mathcal R_2^{k-1} f =0$ which forces $f =0$. Similar arguments work for \eqref{eq:L2HRnewodd}.
\end{proof}

\section{Weighted $L^p-$Hardy--Rellich type inequalities}
In this section, we establish the $L^p-$weighted Hardy--Rellich type inequalities with $p>1$, i.e., $L^p-$version of the inequalities obtained in Section $3$. Throughout this section, we will use the following notation for $p>1$ and $\al \in \R$
\begin{equation}\label{eq:constant}
c_{p,\al} = \frac{(Q -2p-p\al) (Q + p'\al)}{p p'},\quad p' =\frac p{p-1},\quad\text{and}\quad d_{p,\al} = \frac{Q-p(1+\al)}p.
\end{equation}
Let $\xi, \eta \in \C$, denote 
\[
R_p(\xi,\eta) = \frac1p |\eta|^p + \frac{p-1}p |\xi|^p - \Re(|\xi|^{p-2} \xi \overline{\eta}).
\]
By the convexity of $z \to |z|^p$, we see that $R_p(\xi,\eta) \geq 0$ and $R_p(\xi,\eta) = 0$ if and only if $\xi =\eta$. If $\xi, \eta \in \R$, we then have
\[
R_p(\xi, \eta) =(p-1) \int_0^1|t \xi + (1-t)\eta|^{p-2}\, t\, dt\, |\xi-\eta|^2.
\]
We start by establish a weighted version of \eqref{eq:RSHardy}, i.e., a weighted $L^p-$Hardy inequality on $\mathbb G$ which will be frequently used in this section. 

\begin{theorem}\label{LpHardy}
Let $\mathbb G$ be a homogeneous group of homogeneous dimension $Q$. Let $|\cdot|$ be any homogeneous quasi-norm on $\mathbb G$. Let $1 < p < Q$, and for any complex-valued function $f \in C_0^\infty(\mathbb G\setminus\{0\})$, we have 
\begin{equation}\label{eq:identityLpHardy}
\int_{\mathbb G} \frac{|\mathcal R f|^p}{|x|^{p\al}} dx = |d_{p,\al}|^p \int_{\mathbb G} \frac{|f|^p}{|x|^{p(1+ \al)}} dx + p\int_{\mathbb G} \frac1{|x|^{p\al}}R_p\lt(-d_{p,\al}\frac{f}{|x|},\mathcal R f \rt) dx.
\end{equation}
for any $\al \in \R$. As consequence, we obtain the $L^p-$weighted Hardy inequality
\begin{equation}\label{eq:LpweightedHardy}
|d_{p,\al}|^p \int_{\mathbb G} \frac{|f|^p}{|x|^{p(1+ \al)}} dx \leq \int_{\mathbb G} \frac{|\mathcal R f|^p}{|x|^{p\al}} dx,
\end{equation}
for any $\al \in \R$ such that $Q -p(1+\al)\not=0$ and any complex-valued function $f\in C^\infty_0(\mathbb G\setminus\{0\})$. The constant $|d_{p,\al}|^p$ is sharp and equality holds in \eqref{eq:LpweightedHardy} if and only if $f=0$.
\end{theorem}
The case $\al =0$, Theorem \ref{LpHardy} recovers Theorem $3.1$ in \cite{RS2017}. The case $p=2$ Theorem \ref{LpHardy} recovers Theorem \ref{identityHardy} of Ruzhansky and Suragan.
\begin{proof}
Using the polar coordinate \eqref{eq:polar}, we have
\begin{align}\label{eq:LpHardy1}
\int_{\mathbb G} \frac{|f|^p}{|x|^{p(1+ \al)}} dx &= \int_0^\infty r^{Q-p(1+\al)-1} \int_{\mathfrak S} |f(r y)|^p d\si(y) dr\notag\\
&=\frac1{Q-p(1+\al)} \int_0^\infty (r^{Q-p(1+\al)})' \int_{\mathfrak S} |f(r y)|^p d\si(y) dr\notag\\
&=-\frac{p}{Q-p(1+\al)}\Re \int_0^\infty r^{Q-p(1+\al)} \int_{\mathfrak S} |f(r y)|^{p-2} f(ry) \overline{\mathcal R f(ry)} d\si(y) dr\notag\\
&=-\frac{p}{Q-p(1+\al)}\Re \int_{\mathbb G}\frac{|f(ry)|^{p-2} f(ry)}{|x|^{(p-1)(1+\al)}} \frac{\overline{\mathcal R f}}{|x|^{\al}} dx\notag\\
&=\frac{p-1}p \int_{\mathbb G} \frac{|f|^p}{|x|^{p(1+\al)}} dx + \frac1p \lt(\frac{p}{|Q-p(1+\al)|}\rt)^p\int_{\mathbb G} \frac{|\mathcal R f|^p}{|x|^{p\al}} dx\notag\\
&\qquad -\int_{\mathbb G} R_p\lt(\frac{f}{|x|^{1+\al}},-\frac{p}{Q-p(1+\al)}\frac{\mathcal R f}{|x|^\al} \rt) dx.
\end{align}
The equality \eqref{eq:identityLpHardy} is now derived from \eqref{eq:LpHardy1}.

The inequality \eqref{eq:LpweightedHardy} is an immediate consequence of \eqref{eq:identityLpHardy}. The sharpness of \eqref{eq:LpweightedHardy} is verified by testing the approximations of the function $ r^{-(Q-p(1+\al))/p}$. From \eqref{eq:identityLpHardy} we see that equality holds in \eqref{eq:LpweightedHardy} if and only if 
\[
R_p\lt(-\frac{Q-p(1+\al)}p\frac{f}{|x|},\mathcal R f \rt)=0
\]
or equivalently
\[
\mathcal R f = -\frac{Q-p(1+\al)}{p} \frac{f}{|x|}\quad \Longleftrightarrow \quad \mathbb E f = -\frac{Q-p(1+\al)}{p} f.
\]
By Lemma \ref{homo}, $f$ is positively homogeneous of order $-(Q-p(1+\al))/p$ which forces $f =0$ since $|f|/ |x|^{1+\al}$ is in $L^p(\mathbb G)$.
\end{proof}
We next prove a weighted $L^p-$Rellich inequality on $\mathbb G$ which generalizes the inequality \eqref{eq:RellichGRS} to the weighted version and for any $p\in (1,Q/2)$ and generalizes the inequality \eqref{eq:LpRellichE} to the setting of homogeneous groups.

\begin{theorem}\label{LpRellich}
Let $\mathbb G$ be a homogeneous group of homogeneous dimension $Q$. Let $|\cdot|$ be any homogeneous quasi-norm on $\mathbb G$. Let $1 < p < Q/2$, and for any complex-valued function $f \in C_0^\infty(\mathbb G\setminus\{0\})$, we have 
\begin{align}\label{eq:identityLpRellich}
\int_{\mathbb G} \frac{|\mathcal R_2 f|^p}{|x|^{p\al}} dx &= |c_{p,\al}|^p \int_{\mathbb G} \frac{|f|^p}{|x|^{p(2+ \al)}} dx + p\int_{\mathbb G} \frac1{|x|^{p\al}} R_p\lt(c_{p,\al}\frac{f}{|x|^2}, -\mathcal R_2 f\rt) dx \notag\\
&\quad +p |c_{p,\al}|^{p-2}c_{p,\al}\Bigg((p-1)\int_{\mathbb G} \frac{|f|^{p-2}}{|x|^{p(2+\al) -2}} \lt|\mathcal R|f|  + \frac{Q-p(2+\al)}{p|x|} |f|\rt|^2 dx\notag\\
&\qquad\qquad\qquad \qquad \qquad\qquad +\int_{\mathbb G} \frac{|f|^{p-4} (\Im(f \overline{\mathcal R f}))^2}{|x|^{p(2+\al) -2}} dx\Bigg),
\end{align}
for any $\al \in \R$, here $\Im z$ denotes the imagine part of a complex number $z \in \C$. As a consequence, we obtain a weighted $L^p-$Rellich inequality in $\G$ for any $\al \in (-(p-1)Q/p, (Q-2p)/p)$ and any complex-valued function $f\in C^\infty_0(\mathbb G\setminus\{0\})$ as follows
\begin{equation}\label{eq:LpweightedRellich}
c_{p,\al}^p \int_{\mathbb G} \frac{|f|^p}{|x|^{p(2+ \al)}} dx \leq \int_{\mathbb G} \frac{|\mathcal R_2 f|^p}{|x|^{p\al}} dx.
\end{equation}
Moreover, the constant $c_{p,\al}^p$ is sharp and equality holds in \eqref{eq:LpweightedRellich} if and only if $f=0$.
\end{theorem}
\begin{proof}
Using the polar coordinate \eqref{eq:polar}, we have
\begin{align}\label{eq:LpR1}
&\frac{(Q-p(2+\al))(Q-p(2+\al)+1)}p\int_{\mathbb G} \frac{|f|^p}{|x|^{p(2+ \al)}} dx\notag\\
&=\frac{(Q-p(2+\al))(Q-p(2+\al)+1)}p \int_0^\infty r^{Q-p(2+\al)-1} \int_{\mathfrak S} |f(ry)|^p d\si(y) dr\notag\\
&= \frac{(Q-p(2+\al)+1)}p \int_0^\infty (r^{Q-p(2+\al)})'\int_{\mathfrak S} |f(ry)|^p d\si(y) dr\notag\\
&=-(Q-p(2+\al)+1) \Re\int_0^\infty r^{Q-p(2+\al)} \int_{\mathfrak S} |f(ry)|^{p-2} f(ry) \overline{\mathcal R f(ry)} d\si(y) dr\notag\\
&=- \Re\int_0^\infty (r^{Q-p(2+\al)+1})' \int_{\mathfrak S} |f(ry)|^{p-2} f(ry) \overline{\mathcal R f(ry)} d\si(y) dr\notag\\
&= \Re\int_0^\infty r^{Q-p(2+\al)+1} \int_{\mathfrak S} \Big((p-2) |f(ry)|^{p-4} (\Re(f(ry)\overline{\mathcal R f(ry)}))^2 + |f(ry)|^{p-2}  |\mathcal R f(ry)|^2 \notag\\
&\qquad\qquad\qquad\qquad \qquad \qquad\qquad + |f(ry)|^{p-2} f(ry) \overline{\mathcal R^2 f(ry)}\Big) d\si(y) dr\notag\\
&=\Re \int_{\mathbb G} \frac1{|x|^{p(2+\al) -2}} \lt((p-2) |f|^{p-4} (\Re(f \overline{\mathcal R f}))^2 + |f|^{p-2} |\mathcal R f|^2 + |f|^{p-2} f \overline{\mathcal R^2 f}\rt) dx \notag\\
&=\Re \int_{\mathbb G} \frac{|f|^{p-2} f}{|x|^{(p-1)(2+\al)}} \frac{\overline{\mathcal R_2 f}}{|x|^{\al}} dx -(Q-1) \Re \int_{\mathbb G} \frac{|f|^{p-2} f \overline{\mathcal R f}}{|x|^{p(2+\al)-1}} dx \notag\\
&\qquad\qquad\qquad + (p-1) \int_{\mathbb G} \frac{|f|^{p-4} (\Re(f \overline{\mathcal R f}))^2}{|x|^{p(2+\al) -2}} dx + \int_{\mathbb G} \frac{|f|^{p-4} (\Im(f \overline{\mathcal R f}))^2}{|x|^{p(2+\al) -2}} dx\notag\\
&=\Re \int_{\mathbb G} \frac{|f|^{p-2} f}{|x|^{(p-1)(2+\al)}} \frac{\overline{\mathcal R_2 f}}{|x|^{\al}} dx +\frac{(Q-1)(Q-p(2+\al))}p \int_{\mathbb G} \frac{|f|^{p}}{|x|^{p(2+\al)}} dx \notag\\
&\quad +  \frac{4(p-1)}{p^2}\int_{\mathbb G} \frac{(\mathcal R (|f|^{\frac p2}))^2}{|x|^{p(2+\al) -2}} dx + \int_{\mathbb G} \frac{|f|^{p-4} (\Im(f \overline{\mathcal R f}))^2}{|x|^{p(2+\al) -2}} dx,
\end{align}
here we use integration by parts for the last equality. Applying Theorem \ref{identityHardy} for $|f|^{\frac p2}$, we have
\begin{align*}
\int_{\mathbb G} \frac{(\mathcal R (|f|^{\frac p2}))^2}{|x|^{p(2+\al) -2}} dx& =\frac{(Q-p(2+\al))^2}4 \int_{\mathbb G} \frac{|f|^{p}}{|x|^{p(2+\al)}} dx \\
&\qquad\qquad +  \int_{\mathbb G} \frac1{|x|^{p(2+\al) -2}} \lt|\mathcal R|f|^{\frac p2} + \frac{Q-p(2+\al)}{2|x|} |f|^{\frac p2}\rt|^2 dx.
\end{align*}
Plugging this equality into \eqref{eq:LpR1}, we obtain
\begin{align}\label{eq:LpR2}
&\int_{\mathbb G} \frac{|f|^p}{|x|^{p(2+ \al)}} dx\notag\\
&=-\frac{1}{c_{p,\al}}\Re \int_{\mathbb G} \frac{|f|^{p-2} f}{|x|^{(p-1)(2+\al)}} \frac{\overline{\mathcal R_2 f}}{|x|^{\al}} dx -\frac1{c_{p,\al}}\int_{\mathbb G} \frac{|f|^{p-4} (\Im(f \overline{\mathcal R f}))^2}{|x|^{p(2+\al) -2}} dx\notag\\
&\qquad -\frac1{c_{p,\al}} \frac{4(p-1)}{p^2}\int_{\mathbb G} \frac1{|x|^{p(2+\al) -2}} \lt|\mathcal R|f|^{\frac p2} + \frac{Q-p(2+\al)}{2|x|} |f|^{\frac p2}\rt|^2 dx\notag\\
&=\frac{p-1}p \int_{\mathbb G} \frac{|f|^p}{|x|^{p(2+ \al)}} dx + \frac1p \frac1{|c_{p,\al}|^p} \int_{\mathbb G} \frac{|\mathcal R_2 f|^p}{|x|^{p\al}} dx -\int_{\mathbb G} R_p\lt(\frac{f}{|x|^{2+\al}}, -\frac1{c_{p,\al}} \frac{\mathcal R_2 f}{|x|^\al}\rt) dx\notag\\
&\quad -\frac{4(p-1)}{p^2c_{p,\al}}\int_{\mathbb G} \frac1{|x|^{p(2+\al) -2}} \lt|\mathcal R|f|^{\frac p2} + \frac{Q-p(2+\al)}{2|x|} |f|^{\frac p2}\rt|^2 dx\notag\\
&\quad -\frac1{c_{p,\al}}\int_{\mathbb G} \frac{|f|^{p-4} (\Im(f \overline{\mathcal R f}))^2}{|x|^{p(2+\al) -2}} dx.
\end{align}
The equality \eqref{eq:identityLpRellich} is now derived from \eqref{eq:LpR2} and the fact 
\[
\mathcal R (|f|^{\frac p2})  = \frac p2 |f|^{\frac p2-1} \mathcal R (|f|).
\]
The inequality \eqref{eq:LpweightedRellich} is a direct consequence of \eqref{eq:identityLpRellich} since $c_{p,\al} >0$ under the condition of $\al$. The sharpness of \eqref{eq:LpweightedRellich} is verified by using approximations of the function $ r^{-(Q-p(2+\al))/p}$. Suppose that there is equality in \eqref{eq:LpweightedRellich} for some function $f$. From \eqref{eq:identityLpRellich}, we must have
\[
\mathcal R (|f|) + \frac{Q-p(2+\al)}{p|x|} |f| =0 \quad \Longleftrightarrow \mathbb E(|f|) = -\frac{Q-p(2+\al)}{p} |f|.
\]
Using Lemma \ref{homo} implies that $|f|$ is positively homogeneous of degree $-(Q-p(2+\al))/p$ which forces $f = 0$ since $|f|/|x|^{2+\al}$ is in $L^p(\mathbb G)$.
\end{proof}
We next use Theorem \ref{LpHardy} and Theorem \ref{LpRellich} to establish the higher order versions of the weighted $L^p-$Hardy and weighted $L^p-$Rellich inequalities. It gives an $L^p$ analogue of the weighted $L^2-$Hardy--Rellich type inequalities obtained in Theorem \ref{higherorderRellich}. It also gives the generalization of the inequalities \eqref{eq:Lphighereven} and \eqref{eq:Lphigherodd} to the homogeneous groups. In order to do this, we first establish the following interesting equalities.
\begin{proposition}\label{higherorderLp}
Let $\mathbb G$ be a homogeneous group of homogeneous dimension $Q$. Let $|\cdot|$ be any homogeneous quasi-norm on $\mathbb G$. Let $k$ be a positive integer, $1 < p < Q/k$ and let $\al$ be any real number. For any complex-valued function $f \in C_0^\infty(\mathbb G\setminus\{0\})$, we have
\begin{align}\label{eq:Lpeven}
&\int_{\mathbb G} \frac{|\mathcal R_2^l f|^p}{|x|^{p\al}} dx\notag\\
&=\lt|\prod_{i=0}^{l-1}c_{p,2i+ \al}\rt|^p\int_{\mathbb G} \frac{|f|^p}{|x|^{p(k+\al)}} dx + p\int_{\mathbb G} \frac1{|x|^{p\al}} R_p\lt(c_{p,\al}\frac{\mathcal R_2^{l-1} f}{|x|^2}, -\mathcal R_2^l f\rt) dx \notag\\
&\quad + p \sum_{j=1}^{l-1} \lt|\prod_{i=0}^{j-1} c_{p,2i+\al}\rt|^p\int_{\mathbb G} \frac1{|x|^{p(2j+ \al)}} R_p\lt(c_{p,2j+\al}\frac{\mathcal R_2^{l-j-1} f}{|x|^2}, -\mathcal R_2^{l-j} f\rt) dx \notag\\
&\quad +p |c_{p,\al}|^{p-2}c_{p,\al}\Bigg[(p-1)\int_{\mathbb G} \frac{|\mathcal R_2^{l-1}f|^{p-2}}{|x|^{p(2+\al) -2}} \lt|\mathcal R|\mathcal R_2^{l-1}f|  + \frac{Q-p(2+\al)}{p|x|} |\mathcal R_2^{l-1}f|\rt|^2 dx\notag\\
&\qquad\qquad\qquad \qquad \qquad  +\int_{\mathbb G} \frac{|\mathcal R_2^{l-1}f|^{p-4} (\Im(\mathcal R_2^{l-1}f \,\overline{\mathcal R \mathcal R_2^{l-1}f}))^2}{|x|^{p(2+\al) -2}} dx\Bigg]\notag\\
&\quad + p \sum_{j=1}^{l-1}\lt|\prod_{i=0}^{j-1} c_{p,2i+\al}\rt|^p |c_{p,2j+\al}|^{p-2}c_{p,2j+\al} \Bigg[\int_{\mathbb G} \frac{|\mathcal R_2^{l-j-1}f|^{p-4} (\Im(\mathcal R_2^{l-j-1}f \,\overline{\mathcal R \mathcal R_2^{l-j-1}f}))^2}{|x|^{p(2(j+1)+\al) -2}} dx\notag\\
&\quad+ (p-1)\int_{\mathbb G} \frac{|\mathcal R_2^{l-j-1}f|^{p-2}}{|x|^{p(2(j+1)+\al) -2}} \lt|\mathcal R|\mathcal R_2^{l-j-1}f|  + \frac{Q-p(2(j+1) +\al)}{p|x|} |\mathcal R_2^{l-j-1}f|\rt|^2 dx\Bigg].
\end{align}
if $k =2l, l\geq 2$,  and
\begin{align}\label{eq:Lpodd}
&\int_{\mathbb G} \frac{|\mathcal R(\mathcal R_2^l f)|^p}{|x|^{p\al}} dx \notag\\
&=|d_{p,\al}|^p\lt|\prod_{i=0}^{l-1}c_{p,2i+ 1+ \al}\rt|^p\int_{\mathbb G} \frac{|f|^p}{|x|^{p(k+\al)}} dx + p \int_{\mathbb G}\frac1{|x|^{p\al}} R_p\lt(-d_{p,\al}\frac{\mathcal R_2^l f}{|x|}, \mathcal R \mathcal R_2^l f\rt) dx\notag\\
&\quad + p|d_{p,\al}|^p  \int_{\mathbb G} \frac1{|x|^{p(1+\al)}} R_p\lt(c_{p,1+\al}\frac{\mathcal R_2^{l-1} f}{|x|^2}, -\mathcal R_2^l f\rt) dx \notag\\
&\quad + p|d_{p,\al}|^p \sum_{j=1}^{l-1} \lt|\prod_{i=0}^{j-1} c_{p,2i+1+\al}\rt|^p\int_{\mathbb G} \frac1{|x|^{p(2j+ 1+\al)}} R_p\lt(c_{p,2j+1+ \al}\frac{\mathcal R_2^{l-j-1} f}{|x|^2}, -\mathcal R_2^{l-j} f\rt) dx \notag\\
&\quad +p |d_{p,\al}|^p |c_{p,1+\al}|^{p-2}c_{p,1+\al}\Bigg[(p-1)\int_{\mathbb G} \frac{|\mathcal R_2^{l-1}f|^{p-2}}{|x|^{p(3+\al) -2}} \lt|\mathcal R|\mathcal R_2^{l-1}f|  + \frac{Q-p(3+\al)}{p|x|} |\mathcal R_2^{l-1}f|\rt|^2 dx\notag\\
&\qquad\qquad\qquad \qquad \qquad \qquad\qquad +\int_{\mathbb G} \frac{|\mathcal R_2^{l-1}f|^{p-4} (\Im(\mathcal R_2^{l-1}f \,\overline{\mathcal R \mathcal R_2^{l-1}f}))^2}{|x|^{p(3+\al) -2}} dx\Bigg]\notag\\
&\quad + p |d_{p,\al}|^p\sum_{j=1}^{l-1}\lt|\prod_{i=0}^{j-1} c_{p,2i+1+\al}\rt|^p |c_{p,2j+1+\al}|^{p-2}c_{p,2j+1+\al} \times \notag\\
&\qquad \qquad \times \Bigg[\int_{\mathbb G} \frac{|\mathcal R_2^{l-j-1}f|^{p-4} (\Im(\mathcal R_2^{l-j-1}f \,\overline{\mathcal R \mathcal R_2^{l-j-1}f}))^2}{|x|^{p(2j+3 +\al) -2}} dx\notag\\
&\quad\,+ (p-1)\int_{\mathbb G} \frac{|\mathcal R_2^{l-j-1}f|^{p-2}}{|x|^{p(2j+3+\al) -2}} \lt|\mathcal R|\mathcal R_2^{l-j-1}f|  + \frac{Q-p(2j+3 +\al)}{p|x|} |\mathcal R_2^{l-j-1}f|\rt|^2 dx\Bigg].
\end{align}
if $k =2l+1, l\geq 1$ with remark that the terms concerning the sum from $1$ to $l-1$ do not appear if $l=1$. 
\end{proposition}
\begin{proof}
The equality \eqref{eq:Lpeven} is proved by using consecutively \eqref{eq:identityLpRellich}. The equality \eqref{eq:Lpodd} is consequence of \eqref{eq:Lpeven} and \eqref{eq:identityLpHardy}.
\end{proof}
By dropping the nonnegative remainder terms in \eqref{eq:Lpeven} and \eqref{eq:Lpodd}, we obtain the following higher order weighted $L^p-$Hardy-Rellich type inequalities.
\begin{theorem}\label{eq:HRhigher}
Let $G$ be a homogeneous group of homogeneous dimension $Q$. Let $|\cdot|$ be any homogeneous quasi-norm on $\mathbb G$ and let $\al$ be any real number. For any complex-valued function $f \in C_0^\infty(\mathbb G\setminus\{0\})$, we have
\begin{equation}\label{eq:HRLpeven}
\lt(\prod_{i=0}^{k-1}c_{p,2i+ \al}\rt)^p\int_{\mathbb G} \frac{|f|^p}{|x|^{p(2k+\al)}} dx \leq \int_{\mathbb G} \frac{|\mathcal R_2^k f|^p}{|x|^{p\al}} dx,
\end{equation}
if $1 < p < Q/2k$ and $\al \in (-Q(p-1)/p,(Q-2pk)/p)$, and 
\begin{equation}\label{eq:HRLpodd}
d_{p,\al}^p \lt(\prod_{i=0}^{k-1}c_{p,2i+1+ \al}\rt)^p\int_{\mathbb G} \frac{|f|^p}{|x|^{p(2k+1+\al)}} dx \leq \int_{\mathbb G} \frac{|\mathcal R \mathcal R_2^k f|^p}{|x|^{p\al}} dx
\end{equation}
if $1< p< Q/(2k+1)$ and $\al \in (-(Q+p')/p' , (Q-p(2k+1))/p)$. The inequalities \eqref{eq:HRLpeven} and \eqref{eq:HRLpodd} are sharp and equality holds if and only if $f =0$.
\end{theorem}
\begin{proof}
The inequalities \eqref{eq:HRLpeven} and \eqref{eq:HRLpodd} are evidently consequences of  \eqref{eq:Lpeven} and \eqref{eq:Lpeven} respectively since $c_{p,2i+\al} \geq 0$ and $c_{p,2i+1+\al} \geq 0$ for $0\leq i \leq k-1$ in this case. The sharpness of the \eqref{eq:HRLpeven} and \eqref{eq:HRLpodd} is verified by approximating the function $r^{-(Q-p(2k+\al))/p}$ and 
 the function $r^{-(Q-p(2k+1+\al))/p}$ respectively where. Moreover, if equality holds for a function $f$, then by Proposition \ref{higherorderLp}, we must have that $|f|$ is positive homogeneous of degree $-(Q-p(2k+\al))/p$ (corresponding to \eqref{eq:HRLpeven}) and $-(Q-p(2k+1+\al))/p$ (corresponding to \eqref{eq:HRLpodd}) which force $f=0$ by the condition of $L^p-$integrability.
\end{proof}

Theorem \ref{eq:HRhigher} implies the following uncertainly principles.

\begin{corollary}\label{uncertainly1}
Let $k$ be a positive integer and $p >1$. Let $\mathbb G$ be a homogeneous group of homogeneous dimension $Q > kp$, and let $|\cdot|$ be any homogeneous quasi-norm on $\mathbb G$. Then for any complex-valued function $f\in C_0^\infty(\mathbb G \setminus\{0\})$, we have
\begin{equation}\label{eq:uncertaineven}
\lt(\prod_{i=0}^{l-1} c_{p,2i+\al}\rt) \int_{\mathbb G} |f|^2 dx \leq \lt(\int_{\mathbb G} \frac{|\mathcal R_2^l f|^p}{|x|^{p\al}} dx\rt)^{\frac1p} \lt(\int_{\mathbb G} |f|^{p'} |x|^{p'(2l+\al)} dx\rt)^{\frac1{p'}}
\end{equation}
if $k =2l$, $l \geq 1$ and for any $\al \in (-Q(p-1)/p,(Q-2pl)/p)$, and
\begin{equation}\label{eq:uncertainodd}
\lt|\frac{d_{p,\al}}{c_{p,2l+1+\al}} \prod_{i=0}^{l} c_{p,2i+1+\al}\rt| \int_{\mathbb G} |f|^2 dx \leq \lt(\int_{\mathbb G} \frac{|\mathcal R \mathcal R_2^l f|^p}{|x|^{p\al}} dx\rt)^{\frac1p} \lt(\int_{\mathbb G} |f|^{p'} |x|^{p'(2l+1+\al)} dx\rt)^{\frac1{p'}}
\end{equation}
if $k=2l + 1$, $l\geq 0$ and for any $\al \in (-(Q+p')/p' , (Q-p(2l+1))/p)$ if $l\geq 1$ and for any $\al \in \R$ if $l=0$.
\end{corollary}
\begin{proof}
We first prove \eqref{eq:uncertaineven}. By H\"older inequality, we have
\begin{align*}
\int_{\mathbb G} |f|^2 dx &= \int_{\mathbb G} \frac{|f|}{|x|^{2l+\al}} |f| |x|^{2l+\al} dx \leq \lt(\int_{\mathbb G} \frac{|f|^p}{|x|^{p(2l+\al)}}dx\rt)^{\frac1p} \lt(\int_{\mathbb G} |f|^{p'} |x|^{p'(2l+\al)} dx\rt)^{\frac1{p'}}.
\end{align*}
Now, applying the inequality \eqref{eq:HRLpeven} to the previous inequality, we get \eqref{eq:uncertaineven} with remark that $c_{2i+\al} >0$ for $0\leq i\leq l-1$ for $\al \in (-Q(p-1)/p,(Q-2pl)/p)$.

The inequality \eqref{eq:uncertainodd} is proved by the similar way.
\end{proof}
Corollary \ref{uncertainly1} contains an uncertainly principle on the homogeneous group established by Ruzhansky and Suragan \cite[Corollary $3.4$]{RS2017} which corresponds to the case $k=1$ and $\al =0$. For the other cases, the inequalities in Corollary \ref{uncertainly1} seem to be new.

We next establish a $L^p$ version of Theorem \ref{onetwo} for $p\in (1,Q/2)$. To do this, we will need the following equality.

\begin{proposition}\label{onetwoLplemma}
Let $\mathbb G$ be a homogeneous group of homogeneous dimension $Q$. Let $|\cdot|$ be any homogeneous quasi-norm on $\mathbb G$. Let $|\cdot|$ be any homogeneous quasi-norm on $\mathbb G$ and $1< p < Q$. Then for any complex-valued function $f\in C_0^\infty(\mathbb G\setminus\{0\})$, we have
\begin{align}\label{eq:abcd}
\int_{\mathbb G} &\frac1{|x|^{p\al}} \lt|\mathcal R f + \frac{Q-1}r f\rt|^p dx \notag\\
&= \frac{|Q+p'\al|^p}{(p')^p} \int_{\mathbb G} \frac{|f|^p}{|x|^{p(1+\al)}} dx + p \int_{\mathbb G} \frac1{|x|^{p\al}} R_p\lt(\frac{Q+p'\al}{p'} \frac{f}{|x|}, \mathcal R f + \frac{Q-1}{|x|} f\rt) dx
\end{align}
for any $\al \in \R$ with $p' = p/(p-1)$.
\end{proposition}
\begin{proof}
Using the polar coordinate \eqref{eq:polar}, we have
\begin{align*}
\int_{\mathbb G} \frac{|f|^p}{|x|^{p(1+\al)}}& dx \\
&= \int_0^\infty r^{Q-p(1+\al) -1} \int_{\mathfrak S} |f(r y)|^p d\si(y) dr\\
&=\frac1{Q-p(1+ \al)} \int_0^\infty (r^{Q-p(1+ \al)})'\int_{\mathfrak S} |f(r y)|^p d\si(y) dr\\
&= -\frac p{Q-p(1+\al)} \Re \int_0^\infty r^{Q-p(1+\al)} \int_{\mathfrak S} |f(ry)|^{p-2} f(ry) \, \overline{\mathcal R f(ry)} d\si(y) dy\\
&=-\frac p{Q-p(1+\al)} \Re \int_{\mathbb G} \frac{|f|^{p-2} f\, \overline{\mathcal R f}}{|x|^{p(1+ \al) -1}} dx\\
&= -\frac p{Q-p(1+\al)}\lt( \Re \int_{\mathbb G} \frac{|f|^{p-2} f}{|x|^{(p-1)(1+\al)}} \frac{\overline{\mathcal R f + \frac{Q-1}{|x|} f}}{|x|^{\al}} - (Q-1)\int_{\mathbb G} \frac{|f|^p}{|x|^{p(1+\al)}} dx\rt).
\end{align*}
This equality implies
\begin{align*}
\int_{\mathbb G} \frac{|f|^p}{|x|^{p(1+\al)}} dx &= \frac{p'}{Q + p'\al}\Re \int_{\mathbb G} \frac{|f|^{p-2} f}{|x|^{(p-1)(1+\al)}} \frac{\overline{\mathcal R f + \frac{Q-1}{|x|} f}}{|x|^{\al}}\\
&= \frac{p-1}p  \int_{\mathbb G} \frac{|f|^p}{|x|^{p(1+\al)}} dx + \frac1p \frac{(p')^p}{|Q+p'\al|^p}\int_{\mathbb G} \frac1{|x|^{p\al}} \lt|\mathcal R f + \frac{Q-1}{|x|} f\rt|^2 dx\\
&\qquad -\int_{\mathbb G} R_p\lt(\frac{f}{|x|^{1+\al}},\frac{p'}{Q + p'\al}\frac{ \mathcal Rf + \frac{Q-1}{|x|} f}{|x|^{\al}}\rt) dx
\end{align*}
which yields our desired result \eqref{eq:abcd}.
\end{proof}

Using Lemma \ref{onetwoLplemma}, we get the following $L^p$ analogue of Theorem \ref{onetwo}.

\begin{theorem}\label{onetwoLp}
Let $\mathbb G$ be a homogeneous group of homogeneous dimension $Q$. Let $|\cdot|$ be any homogeneous quasi-norm on $\mathbb G$. Let $|\cdot|$ be any homogeneous quasi-norm on $\mathbb G$ and $1< p < Q/2$. Then for any complex-valued function $f\in C_0^\infty(\mathbb G\setminus\{0\})$, we have
\begin{equation}\label{eq:identityLpnew}
\int_{\mathbb G} \frac{|\mathcal R_2 f|^p}{|x|^{p\al}} dx = \frac{|Q+p'\al|^p}{(p')^p} \int_{\mathbb G} \frac{|\mathcal R f|^p}{|x|^{p(1+\al)}} dx + p \int_{\mathbb G} \frac1{|x|^{p\al}} R_p\lt(\frac{Q+p'\al}{p'} \frac{\mathcal R f}{|x|}, \mathcal R_2 f\rt) dx
\end{equation}
for any $\al \in \R$. As consequence, we obtain the following weighted $L^p-$Rellich type inequality
\begin{equation}\label{eq:newLpRellich}
\frac{|Q+p'\al|^p}{(p')^p} \int_{\mathbb G} \frac{|\mathcal R f|^p}{|x|^{p(1+\al)}} dx \leq \int_{\mathcal G} \frac{|\mathcal R_2 f|^p}{|x|^{p\al}} dx,
\end{equation}
for any $\al \in \R$ and for any complex-valued function $f\in C_0^\infty(\mathbb G\setminus\{0\})$. Moreover, the inequality \eqref{eq:newLpRellich} is sharp and equality holds if and only if $f =0$.
\end{theorem}
\begin{proof}
The equality \eqref{eq:identityLpnew} is exactly \eqref{eq:abcd} with $f$ being replaced by $\mathcal R f$. The inequality \eqref{eq:newLpRellich} is immediately implies from \eqref{eq:identityLpnew} by dropping the nonnegative remainder term on the right hand side of \eqref{eq:identityLpnew}. The sharpness of \eqref{eq:newLpRellich} is proved by using approximations of the function $r^{-(Q-p(1+\al))/p}$. If equality occurs in \eqref{eq:newLpRellich} by a function $f$, then by \eqref{eq:identityLpnew}, we must have
\[
\mathcal R_2 f = \frac{Q + p'\al}{p'|x|} \mathcal Rf,
\]
which is equivalent to
\[
\mathcal R^2 f + \frac{Q -p(1+\al)}{p |x|} \mathcal R f=0 \quad \Longleftrightarrow \quad \mathbb E(\mR f) = -\frac{Q -p(1+\al)}{p} \mathcal R f.
\]
Hence $\mathcal R f$ is positively homogeneous of degree $-(Q -p(1+\al))/p$, by Lemma \ref{homo}, which forces $\mathcal R f =0$ since $\mathcal Rf /|x|^{1+\al}$ is in $L^p(\mathbb G)$. Thus, we get $f=0$.
\end{proof}

Combining Theorem \ref{onetwoLp} and Theorem \ref{higherorderLp}, we obtain the following weighted $L^p-$Rellich type inequality which is a $L^p$ analogue of Theorem \ref{L2HRnew}.

\begin{theorem}\label{LpHRnew}
Let $G$ be a homogeneous group of homogeneous dimension $Q$. Let $|\cdot|$ be any homogeneous quasi-norm on $\mathbb G$. Let $k,l$ be nonnegative integers such that $k \geq l+1$ and let $p> 1$. Then for any $\al \in \R$ and for any complex-valued function $f\in C_0^\infty(\mathbb G\setminus\{0\})$, we have
\begin{align}\label{eq:oxoxoxeven}
&\int_{\mathbb G} \frac{|\mathcal R_2^k f|^p}{|x|^{p\al}} dx\notag\\
&=\frac{p^p}{|Q -p \al|^p} \lt|\prod_{i=0}^{k-l-1}\frac{(Q-p(2i+\al))(Q+p'(2i+\al))}{pp'}\rt|^p\int_{\mathbb G} \frac{|\mathcal R \mathcal R_2^{l} f|^p}{|x|^{p(2(k-l)-1+\al)}} dx \notag\\
&\quad + p\lt|\prod_{i=0}^{k-l-2}c_{p,2i+ \al}\rt|^p\int_{\mathbb G} \frac{ R_p\lt(\frac{Q+p'(2(k-l-1)+\al)}{p'} \frac{\mathcal R \mathcal R_2^l f}{|x|}, \mathcal R_2^{l+1} f\rt)}{|x|^{p(2(k-l-1)+ \al)}} dx\notag\\
&\quad + p\int_{\mathbb G} \frac{ R_p\lt(-c_{p,\al}\frac{\mathcal R_2^{k-1} f}{|x|^2}, \mathcal R_2^k f\rt)}{|x|^{p\al}} dx \notag\\
&\quad + p \sum_{j=1}^{k-l-2} \lt|\prod_{i=0}^{j-1} c_{p,2i+\al}\rt|^p\int_{\mathbb G} \frac{ R_p\lt(-c_{p,2j+\al}\frac{\mathcal R_2^{k-j-1} f}{|x|^2}, \mathcal R_2^{k-j} f\rt)}{|x|^{p(2j+ \al)}} dx \notag\\
&\quad +p |c_{p,\al}|^{p-2}c_{p,\al}\Bigg[(p-1)\int_{\mathbb G} \frac{|\mathcal R_2^{k-1}f|^{p-2}}{|x|^{p(2+\al) -2}} \lt|\mathcal R|\mathcal R_2^{k-1}f|  + \frac{Q-p(2+\al)}{p|x|} |\mathcal R_2^{k-1}f|\rt|^2 dx\notag\\
&\qquad\qquad\qquad \qquad \qquad  +\int_{\mathbb G} \frac{|\mathcal R_2^{k-1}f|^{p-4} (\Im(\mathcal R_2^{k-1}f \,\overline{\mathcal R \mathcal R_2^{k-1}f}))^2}{|x|^{p(2+\al) -2}} dx\Bigg]\notag\\
&\quad + p \sum_{j=1}^{k-l-2}\lt|\prod_{i=0}^{j-1} c_{p,2i+\al}\rt|^p |c_{p,2j+\al}|^{p-2}c_{p,2j+\al} \Bigg[\int_{\mathbb G} \frac{|\mathcal R_2^{k-j-1}f|^{p-4} (\Im(\mathcal R_2^{k-j-1}f \,\overline{\mathcal R \mathcal R_2^{k-j-1}f}))^2}{|x|^{p(2(j+1)+\al) -2}} dx\notag\\
&\quad+ (p-1)\int_{\mathbb G} \frac{|\mathcal R_2^{k-j-1}f|^{p-2}}{|x|^{p(2(j+1)+\al) -2}} \lt|\mathcal R|\mathcal R_2^{k-j-1}f|  + \frac{Q-p(2(j+1) +\al)}{p|x|} |\mathcal R_2^{k-j-1}f|\rt|^2 dx\Bigg],
\end{align}
and
\begin{align}\label{eq:oxoxoxodd}
&\int_{\mathbb G} \frac{|\mathcal R \mathcal R_2^k f|^p}{|x|^{p\al}} dx\notag\\
&=\lt|\prod_{i=0}^{k-l-1}\frac{(Q-p(2i+1+\al))(Q+p'(2i+1+\al))}{pp'}\rt|^p\int_{\mathbb G} \frac{|\mathcal R \mathcal R_2^{l} f|^p}{|x|^{p(2(k-l)+\al)}} dx \notag\\&\quad + p|d_{p,\al}|^p\lt|\prod_{i=0}^{k-l-2}c_{p,2i+ 1+\al}\rt|^p\int_{\mathbb G} \frac{ R_p\lt(\frac{Q+p'(2(k-l)-1+\al)}{p'} \frac{\mathcal R \mathcal R_2^l f}{|x|}, \mathcal R_2^{l+1} f\rt)}{|x|^{p(2(k-l)-1+ \al)}} dx\notag\\
&\quad + p|d_{p,\al}|^p\int_{\mathbb G} \frac{ R_p\lt(-c_{p,1+\al}\frac{\mathcal R_2^{k-1} f}{|x|^2}, \mathcal R_2^k f\rt)}{|x|^{p(1+\al)}} dx\notag\\
&\quad  + p |d_{p,\al}|^p\sum_{j=1}^{k-l-2} \lt|\prod_{i=0}^{j-1} c_{p,2i+1+\al}\rt|^p\int_{\mathbb G} \frac{ R_p\lt(-c_{p,2j+1+\al}\frac{\mathcal R_2^{k-j-1} f}{|x|^2}, \mathcal R_2^{k-j} f\rt)}{|x|^{p(2j+1+ \al)}} dx \notag\\
&\quad +p |d_{p,\al}|^p |c_{p,1+\al}|^{p-2}c_{p,1+\al}\Bigg[\int_{\mathbb G} \frac{|\mathcal R_2^{k-1}f|^{p-4} (\Im(\mathcal R_2^{k-1}f \,\overline{\mathcal R \mathcal R_2^{k-1}f}))^2}{|x|^{p(3+\al) -2}} dx\notag\\
&\qquad\qquad\qquad \qquad \qquad  +(p-1)\int_{\mathbb G} \frac{|\mathcal R_2^{k-1}f|^{p-2}}{|x|^{p(3+\al) -2}} \lt|\mathcal R|\mathcal R_2^{k-1}f|  + \frac{Q-p(3+\al)}{p|x|} |\mathcal R_2^{k-1}f|\rt|^2 dx\Bigg]\notag\\
&\quad + p |d_{p,\al}|^p\sum_{j=1}^{k-l-2}\lt|\prod_{i=0}^{j-1} c_{p,2i+1+\al}\rt|^p |c_{p,2j+1+\al}|^{p-2}c_{p,2j+1+\al} \times\notag\\
&\quad\times \Bigg[\int_{\mathbb G} \frac{|\mathcal R_2^{k-j-1}f|^{p-4} (\Im(\mathcal R_2^{k-j-1}f \,\overline{\mathcal R \mathcal R_2^{k-j-1}f}))^2}{|x|^{p(2(j+1)+1+\al) -2}} dx\notag\\
&\quad+ (p-1)\int_{\mathbb G} \frac{|\mathcal R_2^{k-j-1}f|^{p-2}}{|x|^{p(2(j+1)+1+\al) -2}} \lt|\mathcal R|\mathcal R_2^{k-j-1}f|  + \frac{Q-p(2(j+1) +\al)}{p|x|} |\mathcal R_2^{k-j-1}f|\rt|^2 dx\Bigg]\notag\\
&\quad p\int_{\mathbb G} \frac1{|x|^{p\al}}R_p\lt(-\frac{Q-p-p\al}p\frac{\mathcal R_2^k f}{|x|},\mathcal R \mathcal R_2^k f \rt) dx.
\end{align}
As a consequence, the following weighted $L^p-$Rellich type inequalities holds for any function $f\in C_0^\infty(\mathbb G \setminus\{0\})$
\begin{equation}\label{eq:LpHRneweven}
\frac{p^p}{|Q -p \al|^p} \lt|\prod_{i=0}^{k-l-1}\frac{(Q-p(2i+\al))(Q+p'(2i+\al))}{pp'}\rt|^p\int_{\mathbb G} \frac{|\mathcal R \mathcal R_2^{l} f|^p}{|x|^{p(2(k-l)-1+\al)}} dx \leq \int_{\mathbb G} \frac{|\mathcal R_2^k f|^p}{|x|^{p\al}} dx,
\end{equation}
for any $\al \in (-Q(p-1)/p,(Q-2p(k-l-1))/2)$ if $k-l\geq 2$ and for any $\al \in \R$ if $k=l+1$, and
\begin{equation}\label{eq:LpHRnewodd}
\lt|\prod_{i=0}^{k-l-1}\frac{(Q-p(2i+1+\al))(Q+p'(2i+1+\al))}{pp'}\rt|^p\int_{\mathbb G} \frac{|\mathcal R \mathcal R_2^{l} f|^p}{|x|^{p(2(k-l)-1+\al)}} dx \leq \int_{\mathbb G} \frac{|\mathcal R \mathcal R_2^k f|^p}{|x|^{p\al}} dx,
\end{equation}
for any $\al\in (-(Q+p')/p',(Q-p(2(k-l-1)+1))/p)$ if $k-l\geq 2$ and for any $\al \in \R$ if $k=l+1$. Moreover, these inequalities \eqref{eq:LpHRneweven} and \eqref{eq:LpHRnewodd} are sharp and equality holds if and only if $f=0$.
\end{theorem}

\section{The critical Rellich type inequalities}
This section is devoted to study the critical Rellich inequalities. Let us start by recalling a family of logarithmic Hardy inequalities in \cite{RS2016}.

\begin{proposition}\label{RScritical}
Let $\mathbb G$ be a homogeneous group of homogeneous dimension $Q\geq 2$ and a homogeneous quasi-norm denoted by $|\cdot|$. Let $f \in C_0^\infty(\mathbb G\setminus\{0\})$ be any complex-valued function and denote $f_R(x) = f(R x/|x|)$ with $x\in \mathbb G$ and $R >0$. Then we have
\begin{align}\label{eq:RSidentitycritical}
\int_{\mathbb G} \frac{|\mathcal R f|^p}{|x|^{Q-p}} dx &= \lt(\frac{p-1}p\rt)^p \int_{\mathbb G} \frac{|f -f_R|^p}{|x|^Q \lt|\ln \frac R{|x|}\rt|^p} dx \notag\\
&\qquad + p \int_{\mathbb G} \frac{1}{|x|^{Q-p}}R_p\lt(-\frac{p-1}p \frac {f-f_R}{|x| \ln \frac R{|x|}}, \mathcal R f\rt) dx
\end{align}
for any $1 < p < \infty$ and any $R >0$. As a consequence, we obtain the following critical Hardy inequality on $\mathbb G$,
\begin{equation}\label{eq:criticalHR}
\lt(\frac{p-1}p\rt)^p \sup_{R >0} \int_{\mathbb G} \frac{|f -f_R|^p}{|x|^Q \lt|\ln \frac R{|x|}\rt|^p} dx \leq \int_{\mathbb G} \frac{|\mathcal R f|^p}{|x|^{Q-p}} dx,\qquad 1 < p < \infty,
\end{equation}
any complex-valued function $f \in C_0^\infty(\mathbb G\setminus\{0\})$. Moreover, the constant $((p-1)/p)^p$ is sharp.
\end{proposition}
Our main aim of this section is to extend the critical Hardy inequality \eqref{eq:criticalHR} to the higher order derivatives, i.e., to establish the critical Rellich type inequalities on $\mathbb G$. To do this, we first prove a critical Rellich inequality for $\mathcal R_2$ as follows.
\begin{theorem}\label{criticalRellich}
Let $\mathbb G$ be a homogeneous group of homogeneous dimension $Q\geq 3$ and a homogeneous quasi-norm denoted by $|\cdot|$. Let $f \in C_0^\infty(\mathbb G\setminus\{0\})$ be any complex-valued function and denote $f_R(x) = f(R x/|x|)$ with $x\in \mathbb G$ and $R >0$. Then we have
\begin{align}\label{eq:identitycriticalR}
\int_{\mathbb G} \frac{|\mathcal R_2 f|^p}{|x|^{Q-2p}} dx &= \lt(\frac{Q-2}{p'}\rt)^p \int_{\mathbb G} \frac{|f -f_R|^p}{|x|^Q \lt|\ln \frac R{|x|}\rt|^p} dx + p\int_{\mathbb G}\frac{1}{|x|^{Q-2p}} R_p\lt((Q-2)\frac{\mathcal Rf}{|x|}, \mathcal R_2f\rt) dx\notag\\
&\qquad + p(Q-2)^p \int_{\mathbb G} \frac{1}{|x|^{Q-p}}R_p\lt(-\frac{p-1}p \frac {f-f_R}{|x| \ln \frac R{|x|}}, \mathcal R f\rt) dx,
\end{align}
for any $1 < p < \infty$ and any $R >0$. As a consequence, we obtain the following critical Rellich inequality on $\mathbb G$,
\begin{equation}\label{eq:criticalRellich}
\lt(\frac{Q-2}{p'}\rt)^p \sup_{R >0} \int_{\mathbb G} \frac{|f -f_R|^p}{|x|^Q \lt|\ln \frac R{|x|}\rt|^p} dx \leq \int_{\mathbb G} \frac{|\mathcal R_2 f|^p}{|x|^{Q-2p}} dx,\qquad 1 < p < \infty,
\end{equation}
any complex-valued function $f \in C_0^\infty(\mathbb G\setminus\{0\})$. Moreover, the constant $((Q-2)/p')^p$ is sharp.
\end{theorem}
\begin{proof}
It follows from Theorem \ref{onetwoLp} for $\al = (Q-2p)/p$ that
\[
\int_{\mathbb G} \frac{|\mathcal R_2 f|^p}{|x|^{Q-2p}} dx = (Q-2)^p \int_{\mathbb G} \frac{|\mathcal R f|^p}{|x|^{Q-p}} dx + p\int_{\mathbb G}\frac{1}{|x|^{Q-2p}} R_p\lt((Q-2)\frac{\mathcal Rf}{|x|}, \mathcal R_2f\rt) dx 
\]
Plugging \eqref{eq:RSidentitycritical} into the previous equality, we obtain \eqref{eq:identitycriticalR}. The inequality \eqref{eq:criticalRellich} is an immediate consequence of \eqref{eq:identitycriticalR}. It remains to check the sharpness of \eqref{eq:criticalRellich}. For $\ep, \de >0$ small enough and $R >2$, define the function
\[
f_{\de}(x) = \lt(\ln \frac R{|x|}\rt)^{1-\frac1p -\de} \phi(|x|),
\]
where $\phi$ is the function as in the proof of Theorem \ref{identityRellich}. By the direct computations, we have
\begin{align*}
\mathcal R f_{\de}(r)& = -\lt(1-\frac1p -\de\rt)\frac1r \lt(\ln \frac Rr\rt)^{-\frac1p -\de} \phi(r) + \lt(\ln \frac Rr\rt)^{1-\frac1p -\de} \phi'(r) 
\end{align*}
and
\begin{align*}
&\mathcal R^2 f_{\de}(r)\\
&= -2\lt(1-\frac1p -\de\rt)\frac1r \lt(\ln \frac Rr\rt)^{-\frac1p -\de} \phi'(r) -\lt(1-\frac1p -\de\rt)\lt(\frac1p + \de\rt)\frac1 {r^2} \lt(\ln \frac Rr\rt)^{-1-\frac1p -\de} \phi(r)\\
&\quad + \lt(1-\frac1p -\de\rt)\frac1{r^2} \lt(\ln \frac Rr\rt)^{-\frac1p -\de} \phi(r) +  \lt(\ln \frac Rr\rt)^{1-\frac1p -\de} \phi''(r)
\end{align*} 
Thus, we obtain
\begin{align*}
&\mathcal R_2f_{\de}(r) \\
&=-2\lt(1-\frac1p -\de\rt)\frac1r \lt(\ln \frac Rr\rt)^{-\frac1p -\de} \phi'(r) -(Q-2)\lt(1-\frac1p -\de\rt)\frac1{r^2} \lt(\ln \frac Rr\rt)^{-\frac1p -\de} \phi(r)\\
&\quad -\lt(1-\frac1p -\de\rt)\lt(\frac1p + \de\rt)\frac1 {r^2} \lt(\ln \frac Rr\rt)^{-1-\frac1p -\de} \phi(r) +  \lt(\ln \frac Rr\rt)^{1-\frac1p -\de} \mathcal R_2\phi(r).
\end{align*}
Evidently, $(f_{\de})_R =0$ and
\begin{align*}
\int_{\mathbb G} \frac{|f_{\de}-(f_{\de})_R|^p}{|x|^Q \lt|\ln \frac R {|x|}\rt|^p} dx &= \sigma(\mathfrak S) \int_0^2\frac{1}r \lt(\ln R -\ln r\rt)^{-1-\de p } \phi(r)^p dr\\
&\geq \sigma(\mathfrak S)\int_{0}^1 \frac{1}r \lt(\ln R -\ln r\rt)^{-1-\de p } dr\\
&= \frac1{\de p} (\ln R)^{-p\de} |\mathfrak S|.
\end{align*}
Hence 
\[
\lim_{\de \to 0} \int_{\mathbb G} \frac{|f_{\de}-(f_{\de})_R|^p}{|x|^Q \lt|\ln \frac R {|x|}\rt|^p} dx =\infty.
\]
It is easy to check that
\[
\int_{\mathbb G} \frac{1}{|x|^{Q-2p}}\lt|\frac1{|x|} \lt(\ln \frac R{|x|}\rt)^{-\frac1p -\de} \phi'(|x|)\rt|^p dx =O(1),
\]
\[
\int_{\mathbb G} \frac{1}{|x|^{Q-2p}}\lt|\frac1 {|x|^2} \lt(\ln \frac R{|x|}\rt)^{-1-\frac1p -\de} \phi(|x|)\rt|^p dx =O(1),
\]
\[
\int_{\mathbb G} \frac{1}{|x|^{Q-2p}}\lt|\lt(\ln \frac R{|x|}\rt)^{1-\frac1p -\de} \mathcal R_2\phi(|x|)\rt|^p dx =O(1),
\]
and
\begin{align*}
\int_{\mathbb G} \frac{1}{|x|^{Q-2p}}\lt|\frac1{|x|^2} \lt(\ln \frac R{|x|}\rt)^{-\frac1p -\de} \phi(|x|)\rt|^p dx&=\sigma(\mathfrak S)\int_0^2 \frac1 r\lt(\ln R -\ln r\rt)^{-1-p\de} \phi(r) dr\\
&= \int_{\mathbb G} \frac{|f_{\de}-(f_{\de})_R|^p}{|x|^Q \lt|\ln \frac R {|x|}\rt|^p} dx.
\end{align*}
Consequently, we get
\[
\lim_{\de \to 0} \frac{\int_{\mathbb G} \frac{|\mathcal R_2 f_\de|^p}{|x|^{Q-2p}}}{\int_{\mathbb G} \frac{|f_{\de}-(f_{\de})_R|^p}{|x|^Q \lt|\ln \frac R {|x|}\rt|^p} dx} = \lt(\frac{Q-2}{p'}\rt)^p.
\]
This proves the sharpness of \eqref{eq:criticalRellich}. 
\end{proof}
We next combine Theorem \ref{criticalRellich} together with Theorem \ref{LpHardy} and Theorem \ref{higherorderLp} to establish the critical Hardy--Rellich type inequalities of higher orders. Denote 
\[
a_{j,Q} = 2j(Q-2j-2).
\]
Then the following equalities holds true.
\begin{proposition}\label{criticalHRhigher}
Let $\mathbb G$ be a homogeneous group of homogeneous dimension $Q$ and a homogeneous quasi-norm denoted by $|\cdot|$. Let $f \in C_0^\infty(\mathbb G\setminus\{0\})$ be any complex-valued function and denote $f_R(x) = f(R x/|x|)$ with $x\in \mathbb G$ and $R >0$. Then we have for any $R>0$,
\begin{align}\label{eq:criticaleven}
&\int_{\mathbb G} \frac{|\mathcal R_2^k f|^p}{|x|^{Q-2kp}} dx \notag\\
&= \lt(\frac{Q-2}{p'}\rt)^p \lt(\prod_{i=1}^{k-1} a_{i,Q}\rt)^p \int_{\mathbb G} \frac{|f-f_R|^p}{|x|^Q \lt|\ln \frac R{|x|}\rt|^p} dx \notag\\
&\quad + p\lt(\prod_{i=1}^{k-1} a_{i,Q}\rt)^p\int_{\mathbb G}\frac{1}{|x|^{Q-2p}} R_p\lt((Q-2)\frac{\mathcal Rf}{|x|}, \mathcal R_2f\rt) dx\notag\\
&\quad + p(Q-2)^p \lt(\prod_{i=1}^{k-1} a_{i,Q}\rt)^p\int_{\mathbb G} \frac{1}{|x|^{Q-p}}R_p\lt(-\frac{p-1}p \frac {f-f_R}{|x| \ln \frac R{|x|}}, \mathcal R f\rt) dx\notag\\
&\quad + p\int_{\mathbb G} \frac1{|x|^{Q-2kp}} R_p\lt(-a_{k-1,Q}\frac{\mathcal R_2^{k-1} f}{|x|^2}, \mathcal R_2^k f\rt) dx \notag\\
&\quad + p \sum_{j=1}^{k-2} \lt(\prod_{i=k-j}^{k-1} a_{i,Q}\rt)^p\int_{\mathbb G} \frac1{|x|^{Q-2(k-j)p}} R_p\lt(-a_{k-j-1,Q}\frac{\mathcal R_2^{k-j-1} f}{|x|^2}, \mathcal R_2^{k-j} f\rt) dx \notag\\
&\quad +p a_{k-1,Q}^{p-1}\Bigg[(p-1)\int_{\mathbb G} \frac{|\mathcal R_2^{k-1}f|^{p-2}}{|x|^{Q-2(k-1)p -2}} \lt|\mathcal R|\mathcal R_2^{k-1}f|  + \frac{2(k-1)}{|x|} |\mathcal R_2^{k-1}f|\rt|^2 dx\notag\\
&\qquad\qquad\qquad \qquad \qquad  +\int_{\mathbb G} \frac{|\mathcal R_2^{k-1}f|^{p-4} (\Im(\mathcal R_2^{k-1}f \,\overline{\mathcal R \mathcal R_2^{k-1}f}))^2}{|x|^{Q-2(k-1)p -2}} dx\Bigg]\notag\\
&\quad + p \sum_{j=1}^{l-1}\lt(\prod_{i=k-j}^{k-1} a_{k-i-1,Q}\rt)^p a_{k-j-1,Q}^{p-1} \Bigg[\int_{\mathbb G} \frac{|\mathcal R_2^{k-j-1}f|^{p-4} (\Im(\mathcal R_2^{k-j-1}f \,\overline{\mathcal R \mathcal R_2^{k-j-1}f}))^2}{|x|^{Q-2(k-j-1)p -2}} dx\notag\\
&\quad+ (p-1)\int_{\mathbb G} \frac{|\mathcal R_2^{k-j-1}f|^{p-2}}{|x|^{Q-2(k-j-1)p -2}} \lt|\mathcal R|\mathcal R_2^{k-j-1}f|  + \frac{Q-p(2(j+1) +\al)}{p|x|} |\mathcal R_2^{k-j-1}f|\rt|^2 dx\Bigg]
\end{align}
if $2\leq k < Q/2$, and 
\begin{align}\label{eq:criticalodd}
&\int_{\mathbb G} \frac{|\mathcal R \mathcal R_2^k f|^p}{|x|^{Q-(2k+1)p}} dx \notag\\
&= (2k)^p\lt(\frac{Q-2}{p'}\rt)^p \lt(\prod_{i=1}^{k-1} a_{i,Q}\rt)^p \int_{\mathbb G} \frac{|f-f_R|^p}{|x|^Q \lt|\ln \frac R{|x|}\rt|^p} dx \notag\\
&\quad + p(2k)^p\lt(\prod_{i=1}^{k-1} a_{i,Q}\rt)^p\int_{\mathbb G}\frac{1}{|x|^{Q-2p}} R_p\lt((Q-2)\frac{\mathcal Rf}{|x|}, \mathcal R_2f\rt) dx\notag\\
&\quad + p(2k)^p(Q-2)^p \lt(\prod_{i=1}^{k-1} a_{i,Q}\rt)^p\int_{\mathbb G} \frac{1}{|x|^{Q-p}}R_p\lt(-\frac{p-1}p \frac {f-f_R}{|x| \ln \frac R{|x|}}, \mathcal R f\rt) dx\notag\\
&\quad + p(2k)^p\int_{\mathbb G} \frac1{|x|^{Q-2kp}} R_p\lt(-a_{k-1,Q}\frac{\mathcal R_2^{k-1} f}{|x|^2}, \mathcal R_2^k f\rt) dx \notag\\
&\quad + p (2k)^p\sum_{j=1}^{k-2} \lt(\prod_{i=k-j}^{k-1} a_{i,Q}\rt)^p\int_{\mathbb G} \frac1{|x|^{Q-2(k-j)p}} R_p\lt(-a_{k-j-1,Q}\frac{\mathcal R_2^{k-j-1} f}{|x|^2}, \mathcal R_2^{k-j} f\rt) dx \notag\\
&\quad +p(2k)^p a_{k-1,Q}^{p-1}\Bigg[(p-1)\int_{\mathbb G} \frac{|\mathcal R_2^{k-1}f|^{p-2}}{|x|^{Q-2(k-1)p -2}} \lt|\mathcal R|\mathcal R_2^{k-1}f|  + \frac{2(k-1)}{|x|} |\mathcal R_2^{k-1}f|\rt|^2 dx\notag\\
&\qquad\qquad\qquad \qquad \qquad  +\int_{\mathbb G} \frac{|\mathcal R_2^{k-1}f|^{p-4} (\Im(\mathcal R_2^{k-1}f \,\overline{\mathcal R \mathcal R_2^{k-1}f}))^2}{|x|^{Q-2(k-1)p -2}} dx\Bigg]\notag\\
&\quad + p(2k)^p \sum_{j=1}^{k-2}\lt(\prod_{i=k-j}^{k-1} a_{k-i-1,Q}\rt)^p a_{k-j-1,Q}^{p-1} \Bigg[\int_{\mathbb G} \frac{|\mathcal R_2^{k-j-1}f|^{p-4} (\Im(\mathcal R_2^{k-j-1}f \,\overline{\mathcal R \mathcal R_2^{k-j-1}f}))^2}{|x|^{Q-2(k-j-1)p -2}} dx\notag\\
&\quad+ (p-1)\int_{\mathbb G} \frac{|\mathcal R_2^{k-j-1}f|^{p-2}}{|x|^{Q-2(k-j-1)p -2}} \lt|\mathcal R|\mathcal R_2^{k-j-1}f|  + \frac{Q-p(2(j+1) +\al)}{p|x|} |\mathcal R_2^{k-j-1}f|\rt|^2 dx\Bigg]\notag\\
&\quad + p\int_{\mathbb G} \frac1{|x|^{Q-(2k+1)p}}R_p\lt(-2k\frac{\mathcal R_2^k f}{|x|},\mathcal R \mathcal R_2^k f \rt) dx,
\end{align}
if $1\leq k < (Q-1)/2$.
\end{proposition}
\begin{proof}
We first prove \eqref{eq:criticaleven}. Denote $g = \mathcal R_2 f$. Applying \eqref{eq:Lpeven} to the function $g$ and $l =k-1$, and then using \eqref{eq:identitycriticalR}, we obtain \eqref{eq:criticaleven}.

We next prove \eqref{eq:criticalodd}. Using Theorem \ref{LpHardy} with $\al =(Q-(2k+1)p)/p$ we get
\begin{align*}
\int_{\mathbb G} \frac{|\mathcal R \mathcal R_2^k f|^p}{|x|^{Q-(2k+1)p}} dx&=(2k)^p\int_{\mathbb G} \frac{|\mathcal R_2^k f|^p}{|x|^{Q-2kp}} dx +  p\int_{\mathbb G} \frac1{|x|^{Q-(2k+1)p}}R_p\lt(-2k\frac{\mathcal R_2^k f}{|x|},\mathcal R \mathcal R_2^k f \rt) dx.
\end{align*}
Combining the previous equality together with \eqref{eq:criticaleven} implies \eqref{eq:criticalodd}. 
\end{proof}
By dropping the nonnegative remainder terms, Proposition \ref{criticalHRhigher} yields the following critical Hardy--Rellich type inequalities on $\mathbb G$ which are extensions of the critical Hardy inequality \eqref{eq:criticalHR} and the critical Rellich inequality \eqref{eq:criticalRellich} to higher order of derivatives.
\begin{theorem}\label{criticalhigher}
Let $\mathbb G$ be a homogeneous group of homogeneous dimension $Q$ and a homogeneous quasi-norm denoted by $|\cdot|$. Let $f \in C_0^\infty(\mathbb G\setminus\{0\})$ be any complex-valued function and denote $f_R(x) = f(R x/|x|)$ with $x\in \mathbb G$ and $R >0$. Then we have
\begin{equation}\label{eq:criticalRellicheven}
\lt(\frac{2^{k-1}(k-1)!}{p'}\prod_{i=0}^{k-1} (Q-2i-2)\rt)^p \sup_{R >0}\int_{\mathbb G} \frac{|f-f_R|^p}{|x|^Q \lt|\ln \frac R{|x|}\rt|^p} dx \leq \int_{\mathbb G} \frac{|\mathcal R_2^k f|^p}{|x|^{Q-2kp}} dx,
\end{equation}
for any $2 \leq k < Q/2$ and for any $p>1$, and 
\begin{equation}\label{eq:criticalRellichodd}
\lt(\frac{2^k k!}{p'}\prod_{i=0}^{k-1} (Q-2i-2)\rt)^p \sup_{R >0}\int_{\mathbb G} \frac{|f-f_R|^p}{|x|^Q \lt|\ln \frac R{|x|}\rt|^p} dx \leq \int_{\mathbb G} \frac{|\mathcal R \mathcal R_2^k f|^p}{|x|^{Q-(2k+1)p}} dx,
\end{equation}
for any $1\leq k \leq (Q-1)/2$ and for any $p>1$. Moreover, the inequalities \eqref{eq:criticalRellicheven} and \eqref{eq:criticalRellichodd} are sharp.
\end{theorem}
\begin{proof}
Obviously that $a_{i,Q} \geq 0$ for $1\leq i \leq k-1$, hence the inequalities \eqref{eq:criticalRellicheven} and \eqref{eq:criticalRellichodd} are immediate consequence of \eqref{eq:criticaleven} and \eqref{eq:criticalodd} respectively with the remark that
\[
\lt(\frac{Q-2}{p'}\rt)^p \lt(\prod_{i=1}^{k-1} a_{i,Q}\rt)^p = \lt(\frac{2^{k-1}(k-1)!}{p'}\prod_{i=0}^{k-1} (Q-2i-2)\rt)^p,
\]
and
\[
(2k)^p\lt(\frac{Q-2}{p'}\rt)^p \lt(\prod_{i=1}^{k-1} a_{i,Q}\rt)^p = \lt(\frac{2^k k!}{p'}\prod_{i=0}^{k-1} (Q-2i-2)\rt)^p.
\]

To verify the sharpness of \eqref{eq:criticalRellicheven} and \eqref{eq:criticalRellichodd}, we consider the test functions
\[
f_\de(x) = \lt(\ln \frac R {|x|}\rt)^{1-\frac1p -\de} \phi(|x|),
\]
as used in the proof of Theorem \ref{criticalRellich} and make the same computations. We skip the slightly tedious details.
\end{proof}

Theorem \ref{criticalhigher} implies the following uncertainly type principles
\begin{corollary}\label{uncertainly2}
Let $G$ be a homogeneous group of homogeneous dimension $Q$ and let $|\cdot|$ be any homogeneous quasi-norm on $G$. Let $k$ be a positive integer less than $Q$. Then for any $f \in C^\infty_0(\mathbb G\setminus\{0\})$, any $R >0$ and $p,q > 1$ such that $1/p + 1/q =1/2$, we have
\begin{equation}\label{eq:even1}
\frac{2^{l-1}(l-1)!}{p'}\lt(\prod_{i=0}^{l-1} (Q-2i-2)\rt)\lt(\int_{\mathbb G} \frac{|f|^2|f -f_R|^2}{|x|^{\frac {2Q}p} |\ln \frac R{|x|}|^2} dx\rt)^{\frac12} \leq \|f\|_{L^q(\mathbb G)}\lt(\int_{\mathbb G} \frac{|\mathcal R_2^l f|^p}{|x|^{Q-2lp}} dx\rt)^{\frac1p} 
\end{equation}
if $k=2l$, $l\geq 1$ and 
\begin{equation}\label{eq:odd1}
\frac{2^l l!}{p'}\lt(\prod_{i=0}^{l-1} (Q-2i-2)\rt)\int_{\mathbb G} \frac{|f|^2|f -f_R|^2}{|x|^{\frac {2Q}p} |\ln \frac R{|x|}|^2} dx \leq \|f\|_{L^q(\mathbb G)}\lt(\int_{\mathbb G} \frac{|\mathcal R \mathcal R_2^l f|^p}{|x|^{Q-(2l+1)p}} dx\rt)^{\frac1p}
\end{equation}
if $k =2l+1$, $l\geq 0$.

Also, we have for any complex-valued function $f \in C_0^\infty(\mathbb G \setminus\{0\})$
\begin{multline}\label{eq:even2}
\frac{2^{l-1}(l-1)!}{p'}\lt(\prod_{i=0}^{l-1} (Q-2i-2)\rt)\int_{\mathbb G} \frac{|f -f_R|^2}{|x|^{Q} |\ln \frac R{|x|}|^2} dx \\
\leq \lt(\int_{\mathbb G} \frac{|\mathcal R_2^l f|^p}{|x|^{Q-2lp}} dx\rt)^{\frac1p}\lt(\int_{\mathbb G} \frac{|f -f_R|^{p'}}{|x|^{Q} |\ln \frac R{|x|}|^{p'}} dx\rt)^{\frac1{p'}}
\end{multline}
if $k = 2l$, $l\geq 1$, and 
\begin{multline}\label{eq:odd2}
\frac{2^l l!}{p'}\lt(\prod_{i=0}^{l-1} (Q-2i-2)\rt)\int_{\mathbb G} \frac{|f -f_R|^2}{|x|^{Q} |\ln \frac R{|x|}|^2} dx \\
\leq \lt(\int_{\mathbb G} \frac{|\mathcal R \mathcal R_2^l f|^p}{|x|^{Q-(2l+1)p}} dx\rt)^{\frac1p}\lt(\int_{\mathbb G} \frac{|f -f_R|^{p'}}{|x|^{Q} |\ln \frac R{|x|}|^{p'}} dx\rt)^{\frac1{p'}}
\end{multline}
if $k = 2l+1$, $l\geq 0$.
\end{corollary} 
\begin{proof}
These inequalities are consequence of Theorem \ref{criticalhigher} and H\"older inequality.
\end{proof}
Again, Corollary \ref{uncertainly2} contains a uncertainly type principle on the homogeneous group recently proved by Ruzhansky and Suragan \cite[Corollary $3.2$]{RS2016} corresponding to the case $k =1$. Corollary \ref{uncertainly2} provides an extension of their result to the higher order derivative.

\section{The inequalities in the Euclidean space}
We restrict ourselves in this section to the Euclidean space, i.e., the abelian case $\mathbb G = (\R^n,+)$ and $Q =n$. Let $|\cdot|$ denote the usual Euclidean norm on $\R^n$. In this case, $\mathcal R$ is exactly the derivative in the radial direction, i.e., $\pa_r$, and 
\[
\mathcal R_2 = \pa_r^2 + \frac{n-1}r \pa_r=: \Delta_r
\]
to be the radial Laplacian. In the sequel, we collect some inequalities obtained from previous sections when restricting to the Euclidean space $\R^n$.  Firstly, Theorem \ref{higherorderRellich} gives
\begin{align}\label{eq:evenE}
&\lt(\prod_{i=0}^{k-1} c_{2i+ \al}\rt)^2\int_{\mathbb R^n} \frac{|f|^2}{|x|^{4k+2\al}} dx \notag\\
&=\int_{\mathbb R^n} \frac{|\De_r^k f|^2}{|x|^{2\al}} dx -\int_{\mathbb R^n} \frac1{|x|^{2\al}} \lt|\De_r^k f + c_\al \frac{\De_r^{k-1}f}{|x|^2}\rt|^2 dx \notag\\
&\quad -\sum_{j=1}^{k-1}\lt(\prod_{i=0}^{j-1} c_{2i+ \al}\rt)^2 \int_{\mathbb R^n} \frac1{|x|^{4j+ 2\al}} \lt|\De_r^{k-j}f + c_{2j+\al} \frac{\De_r^{k-j-1}f}{|x|^2}\rt|^2 dx\notag\\
&\quad -2c_\al \int_{\mathbb R^n} \lt||x|^{1-\frac n2}\pa_r (|x|^{\frac{n-2\al-4}2}\De_r^{k-1}f)\rt|^2 dx\notag\\
&\quad -2 \sum_{j=1}^{k-1} \lt(\prod_{i=0}^{j-1} c_{2i+ \al}\rt)^2 c_{2j+\al} \int_{\mathbb R^n}\lt||x|^{1-\frac n2}\pa_r (|x|^{\frac{n-2\al-4-4j}2}\De_r^{k-j-1}f)\rt|^2 dx,
\end{align}
and
\begin{align}\label{eq:oddE}
&\lt(\frac{n-2-2\al}2 \prod_{i=0}^{k-1} c_{2i+1+ \al}\rt)^2\int_{\mathbb R^n}\frac{|f|^2}{|x|^{4k+2+ 2\al}} dx\notag\\
&=\int_{\mathbb R^n} \frac{|\pa_r(\De_r^k f)|^2}{|x|^{2\al}} dx -\int_{\mathbb R^n}  \lt||x|^{1-\frac n2}\pa_r(|x|^{\frac{n-2\al-2}2}\De_r^k f) \rt|^2 dx \notag\\
&\quad -\frac{(n-2-2\al)^2}4 \int_{\mathbb R^n} \frac1{|x|^{2+2\al}} \lt|\De_r^k f + c_{1+\al} \frac{\De_r^{k-1}f}{|x|^2}\rt|^2 dx \notag\\
&\quad -\frac{(n-2-2\al)^2}4\sum_{j=1}^{k-1} \lt(\prod_{i=0}^{j-1} c_{2i+ \al}\rt)^2\int_{\mathbb R^n} \frac1{|x|^{2+2\al+ 4j}} \lt|\De_r^{k-j} f + c_{2j+ 1+\al} \frac{\De_r^{k-j-1}f}{|x|^2}\rt|^2 dx \notag\\
&\quad -\frac{(n-2-2\al)^2}2 c_{1+\al}\int_{\mathbb R^n} \lt||x|^{1-\frac n2}\pa_r(|x|^{\frac{n-2\al -6}2}\De_r^{k-1} f)\rt|^2 dx\notag\\
&\quad -\frac{(n-2-2\al)^2}2\sum_{j=1}^{k-1}\lt(\prod_{i=0}^{j-1} c_{2i+ 1+\al}\rt)^2 c_{2j+1+\al}  \int_{\mathbb R^n} \lt||x|^{1-\frac n2}\pa_r (|x|^{\frac{n-2\al-6-4j}2}\De_r^{k-j-1} f)\rt|^2 dx
\end{align}
for any function $f\in C_0^\infty(\R^n\setminus\{0\})$. The equality \eqref{eq:oddE} with $k =\al =0$ was proved by Machihara, Ozawa and Wadade in \cite{MOW2016}. The equality \eqref{eq:evenE} with $k=1$ and $\al =0$ was proved by these same authors in \cite{MOW2017}. We believe that the other cases of \eqref{eq:evenE} and \eqref{eq:oddE} to be new. Evidently, \eqref{eq:evenE} and \eqref{eq:oddE} imply the following weighted $L^2$-Hardy--Rellich type inequalities
\begin{equation}\label{eq:HRevenE}
\lt(\prod_{i=0}^{k-1} \frac{(n+4i+2\al)(n-4i-2\al -4)}4\rt)^2\int_{\mathbb R^n} \frac{|f|^2}{|x|^{4k+2\al}} dx \leq \int_{\mathbb R^n} \frac{|\De_r^k f|^2}{|x|^{2\al}} dx
\end{equation}
for $n \geq 4k+1$, $\al \in (-n/2, (n-4k)/2)$, and
\begin{multline}\label{eq:HRoddE}
\lt(\frac{n-2-2\al}2 \prod_{i=0}^{k-1} \frac{(n+4i+2+2\al)(n-4i-6-2\al)}4\rt)^2\int_{\mathbb R^n}\frac{|f|^2}{|x|^{4k+2+ 2\al}} dx \\
\leq \int_{\mathbb R^n} \frac{|\pa_r(\De_r^k f)|^2}{|x|^{2\al}} dx
\end{multline}
for $n\geq 4k+3$, $\al \in (-(n+2)/2, (n-4k-2)/2)$. It is easy to see that $|\pa_r f|\leq |\na f|$. This together with \eqref{eq:HRoddE} gives a simple proof of the classical $L^2$-Hardy inequality \eqref{eq:L2HardyE}.

For $1\leq j \leq n$, we denote by $L_j$ a spherical derivative, i.e.,
\[
L_j = \pa_j -\frac{x_j}{|x|} \pa_r = \pa_j -\sum_{k=1}^n\frac{x_j x_k \pa_k}{|x|^2}.
\]
It was proved in \cite{MOW2017} that 
\begin{align}\label{eq:MOW}
\|\De f\|_{L^2(\R^n)}^2&= \|\De_r f\|_{L^2(\R^n)}^2 + \lt\|\sum_{j=1}^n L_j^2 f\rt\|_{L^2(\R^n)}^2 + \frac{n(n-4)}2\sum_{j=1}^n \lt\| \frac{1}{|x|}L_jf\rt\|_{L^2(\R^n)}^2 \notag\\
&\quad + 2 \sum_{j=1}^n \lt\|\pa_r L_j f + \frac{n-2}{|x|} L_jf\rt\|_{L^2(\R^n)}^2,
\end{align}
for $n\geq 5$ which implies $\|\De_r f\|_{L^2(\R^n)} \leq \|\De f\|_{L^2(\R^n)}$ with equality holds if and only if $f$ is radial function. Again, this estimate together with \eqref{eq:HRevenE} gives a simple proof of the classical $L^2-$Rellich inequality \eqref{eq:RellichE}. 

The next result gives an expression of $\|\na \De f\|_{L^2(\R^n)}^2$ in the spirit of \eqref{eq:MOW}.

\begin{theorem}\label{Energyidentity}
Let $n\geq 7$. Then the following equality holds for all function $f\in H^3(\R^n)$:
\begin{align}\label{eq:energyidentity}
\int_{\R^n} |\na \Delta f|^2 dx&= \int_{\R^n} \lt| \pa_r\Delta_r f\rt|^2 dx + \int_{\R^n} \lt|\pa_r\lt(\sum_{j=1}^n L_j^2 f\rt)\rt|^2 dx + \sum_{j=1}^n \int_{\R^n} \lt|L_j \Delta f\rt|^2 dx\notag\\
&\qquad +\lt(\frac{n^2-24}2 \frac{(n-4)^2}4  + 2(n-3)(n-7)\rt) \sum_{j=1}^n\int_{\R^n} \frac{|L_j f|^2}{|x|^4} dx \notag\\
&\qquad + \frac{n^2-24}2 \sum_{j=1}^n\int_{\R^n} \lt||x|^{1-\frac n2} \pa_r \lt(|x|^{\frac n2-2} L_j f\rt)\rt|^2 dx\notag\\
&\qquad +2\sum_{j=1}^n\int_{\R^n} \lt| |x|^{1-\frac n2} \pa_r(|x|^{\frac n2 -1} \pa_r L_jf)\rt|^2 dx.
\end{align}
As a consequence, we obtain
\begin{equation}\label{eq:compare}
\int_{\R^n} |\na \Delta f|^2 dx \geq \int_{\R^n} \lt|\pa_r \Delta_r f \rt|^2 dx,
\end{equation}
with equality holds if and only if $f$ is radial function.
\end{theorem}
Combining \eqref{eq:compare} with \eqref{eq:HRoddE} for $k=1$, we obtain a simple proof of the Rellich type inequality for $n\geq 7$
\[
\frac{(n-6)^2(n-2)^2 (n+2)^2}{64} \int_{\R^n}\frac{|f|^2}{|x|^6} dx \leq \int_{\R^n} |\na \De f|^2 dx
\]
which is a special case of \eqref{eq:higherRellichEodd} with $m=1$. It would be interesting to find an analogues of Theorem \ref{Energyidentity} for the higher order derivatives. Such a result combining with \eqref{eq:HRevenE} or \eqref{eq:HRoddE} would give a simple proof of the Hardy--Rellich type inequalities \eqref{eq:higherRellichEeven} and \eqref{eq:higherRellichEodd}. Let us go to the proof of Theorem \ref{Energyidentity}.
\begin{proof}[Proof of Theorem \ref{Energyidentity}]
It is enough to prove \eqref{eq:energyidentity} for function $f\in C_0^\infty(\R^n\setminus\{0\})$ by the density argument. Notice that 
\[
\Delta = \Delta_r + \sum_{j=1}^n L_j^2,
\]
which can be verified by some simple computations. Expanding the scalar product, we have
\begin{align}\label{eq:fi00}
\int_{\R^n} |\na \Delta f|^2 dx &= \int_{\R^n} |\pa_r \Delta f|^2 dx + \sum_{j=1}^n \int_{\R^n} |L_j \Delta f|^2 dx\notag\\
&= \int_{\R^n} \lt|\pa_r\Delta_r f + \pa_r\lt(\sum_{j=1}^n L_j^2 f\rt)\rt|^2 dx + \sum_{j=1}^n \int_{\R^n} |L_j \Delta f|^2 dx\notag\\
&= \int_{\R^n} \lt|\pa_r\Delta_r f\rt|^2 dx + \int_{\R^n} \lt|\pa_r\lt(\sum_{j=1}^n L_j^2 f\rt)\rt|^2 dx \notag\\
&\qquad + 2 {\rm Re} \int_{\R^n} \lt(\pa_r\Delta_r f\rt) \overline{\pa_r\lt(\sum_{j=1}^n L_j^2 f\rt)} dx + \sum_{j=1}^n \int_{\R^n} |L_j \Delta f|^2 dx.
\end{align}
Our next goal is to compute
\[
{\rm Re}\int_{\R^n} \lt(\pa_r\Delta_r f\rt) \overline{\pa_r\lt(\sum_{j=1}^n L_j^2 f\rt)} dx.
\]
Denote $h_j = L_j f, j =1,2,\ldots,n$ and $ g = \Delta_r f$. Using integration by parts, we have

\begin{align}\label{eq:IbP***}
\int_{\R^n} \pa_r g \overline{\pa_r\lt(\sum_{j=1}^n L_j^2 f\rt)} dx&= -\int_{\R^n} g \o{\Delta_r \lt(\sum_{j=1}^n L_j^2 f\rt)} dx.
\end{align}

We can readily check that $\pa_r L_j = L_j\pa_r - L_j/|x|$, and
\[
\De_r L_j = L_j \Delta_r - \frac{2}{|x|}L_j \pa_r - \frac{n-3}{|x|^2} L_j,
\]
Since $\sum_{j=1}^n x_j L_j =0,$ then 
\[
\sum_{j=1}^n x_j \Delta_r L_j = \sum_{j=1}^n x_j \pa_r L_j = 0.
\]
Using integration by parts and the fact $L_j(|x|^a u) =|x|^a L_j u$, we get
\begin{align*}
\int_{\R^n} g \o{\Delta_r \lt(\sum_{j=1}^n L_j^2 f\rt)} dx&= \sum_{j=1}^n \int_{\R^n} g \o{\lt(L_j \Delta_r h_j -\frac{2}{|x|}L_j\pa_r h_j -\frac{n-3}{|x|^2}L_j h_j\rt)} dx\\
&= -\sum_{j=1}^n \int_{\R^2} L_j g \o{\lt(\Delta_r h_j -\frac2{|x|}\pa_r h_j - \frac{n-3}{|x|^2} h_j\rt)} dx.
\end{align*}
Notice that $L_j \Delta_r = \Delta_r L_j + 2\pa_r L_j/|x| + (n-1) L_j/|x|^2$, hence
\[
L_j g=L_j \Delta_r f = (\Delta_r L_j + \frac2{|x|} \pa_r L_j + \frac{n-1}{|x|^2} L_j) f = \Delta_r h_j + \frac2{|x|}\pa_r h_j + \frac{n-1}{|x|^2} h_j.
\]
This expression of $L_j g$ implies 
\begin{align*}
{\rm Re}&\int_{\R^n} g \o{\Delta_r \lt(\sum_{j=1}^n L_j^2 f\rt)} dx\\
&=\sum_{j=1}^n{\rm Re} \int_{\R^n} \lt(\Delta_r h_j + \frac2{|x|}\pa_r h_j + \frac{n-1}{|x|^2} h_j\rt)\o{\lt(\Delta_r h_j -\frac2{|x|}\pa_r h_j - \frac{n-3}{|x|^2} h_j\rt)} dx\\
&=\sum_{j=1}^n\int_{\R^n} \lt(|\Delta_r h_j|^2 - \frac4{|x|^2} |\pa_r h_j|^2 - \frac{(n-1)(n-3)}{|x|^4} |h_j|^2\rt) dx + 2\sum_{j=1}^n {\rm Re} \int_{\R^n} \frac{\Delta_r h_j \o{h_j}}{|x|^2} dx\\
&\qquad\qquad -4(n-2) \sum_{j=1}^n{\rm Re} \int_{\R^n} \frac{\pa_r h_j \o{h_j}}{|x|^3} dx\\
&= \sum_{j=1}^n\int_{\R^n} \lt(|\Delta_r h_j|^2 - \frac4{|x|^2} |\pa_r h_j|^2 - \frac{(n-1)(n-3)}{|x|^4} |h_j|^2\rt) dx + 2\sum_{j=1}^n{\rm Re} \int_{\R^n} \frac{\pa_r^2 h_j \o{h_j}}{|x|^2} dx\\
&\qquad\qquad -2(n-3)\sum_{j=1}^n {\rm Re} \int_{\R^n} \frac{\pa_r h_j \o{h_j}}{|x|^3} dx.
\end{align*}
Using integration by parts, we obtain
\begin{align*}
{\rm Re}\int_{\R^n} g \o{\Delta_r \lt(\sum_{j=1}^n L_j^2 f\rt)} dx&=\sum_{j=1}^n\int_{\R^n} \lt(|\Delta_r h_j|^2 - \frac6{|x|^2} |\pa_r h_j|^2 - \frac{(n-1)(n-3)}{|x|^4} |h_j|^2\rt) dx \\
&\qquad\qquad\qquad -4(n-3)\sum_{j=1}^n {\rm Re} \int_{\R^n} \frac{\pa_r h_j \o{h_j}}{|x|^3} dx\\
&=\sum_{j=1}^n\int_{\R^n} \lt(|\Delta_r h_j|^2 - \frac6{|x|^2} |\pa_r h_j|^2 - \frac{(n-1)(n-3)}{|x|^4} |h_j|^2\rt) dx \\
&\qquad\qquad\qquad + 2(n-3)(n-4)\sum_{j=1}^n \int_{\R^n} \frac{|h_j|^2}{|x|^4} dx\\
&=\sum_{j=1}^n \int_{\R^n} \lt(|\Delta_r h_j|^2 - \frac6{|x|^2} |\pa_r h_j|^2 + \frac{(n-3)(n-7)}{|x|^4} |h_j|^2\rt) dx.
\end{align*}
Applying Theorem \ref{onetwo} to $\mathbb G = (\R^n,+)$ with the Euclidean norm and the integral of $|\Delta_r h_j|^2$, we obtain
\begin{align}\label{eq:oxoxox}
{\rm Re}\int_{\R^n} g \o{\Delta_r \lt(\sum_{j=1}^n L_j^2 f\rt)} dx&= -\frac{n^2-24}4 \sum_{j=1}^n\int_{\R^n} \frac{|\pa_r h_j|^2}{|x|^2} dx -(n-3)(n-7)\sum_{j=1}^n \int_{\R^n} \frac{|h_j|^2}{|x|^4} dx \notag\\
&\qquad - \sum_{j=1}^n\int_{\R^n} \lt| |x|^{1-\frac n2} \pa_r(|x|^{\frac n2 -1} \pa_r h_j)\rt|^2 dx.
\end{align}
Using \eqref{eq:oddE} for $k=0$, $\al =1$, we have
\begin{equation}\label{eq:ixixix}
\int_{\R^n} \frac{|\pa_r h_j|^2}{|x|^2} dx = \frac{(n-4)^2}4 \int_{\R^n} \frac{|h_j|^2}{|x|^4} dx + \int_{\R^n} \lt||x|^{1-\frac n2}\pa_r (|x|^{\frac n2 -2} h_j)\rt|^2 dx.
\end{equation}
Plugging \eqref{eq:ixixix}, \eqref{eq:oxoxox} and \eqref{eq:IbP***} into \eqref{eq:fi00} implies our desired equality \eqref{eq:energyidentity}.

The inequality \eqref{eq:compare} is an immediate consequence of \eqref{eq:energyidentity} since $n\geq 7$. Moreover, if $f$ is radial then equality holds since $L_j f =0$ for all $j =1,2,\ldots,n$. Conversely, suppose that equality holds in \eqref{eq:compare}. Since
\[
\frac{n^2-24}4 \frac{(n-4)^2}4 + 2(n-3)(n-7) >0,
\]
for any $n\geq 7$, then we must have $L_j f =0$ for all $j=1,2,\ldots,n$. In particular, we get
\[
x_k \pa_j f(x) = \frac{x_k x_j}{|x|} \pa_r f(x) = x_j \pa_k f(x),
\]
for any $j,k =1,2,\ldots,n$. This fact implies that $f$ is radial function.
\end{proof}

Next, Theorem \ref{eq:HRhigher} gives the following weighted $L^p-$Hardy--Rellich type inequalities on $\R^n$.
\begin{equation}\label{eq:HRLpevenE}
\lt(\prod_{i=0}^{k-1}\frac{n+p'(2i+\al))(n -p(2i+2+\al))}{pp'}\rt)^p\int_{\mathbb R^n} \frac{|f|^p}{|x|^{p(2k+\al)}} dx \leq \int_{\mathbb R^n} \frac{|\De_r^k f|^p}{|x|^{p\al}} dx,
\end{equation}
if $1 < p < n/(2k)$ and $\al \in (-n(p-1)/p,(n-2pk)/p)$, and 
\begin{multline}\label{eq:HRLpoddE}
\frac{(n-p(1+\al))^p}{p^p} \lt(\prod_{i=0}^{k-1}\frac{n+p'(2i+1+\al))(n-p(2i+3+\al))}{pp'}\rt)^p\int_{\mathbb R^n} \frac{|f|^p}{|x|^{p(2k+1+\al)}} dx \\
\leq \int_{\mathbb R^n} \frac{|\pa_r \De_r^k f|^p}{|x|^{p\al}} dx
\end{multline}
if $1< p< n/(2k+1)$ and $\al \in (-(n+p')/p' , (n-p(2k+1))/p)$. Similarly, Theorem \ref{onetwoLp} gives the following weighted $L^p$-Rellich type inequality
\begin{equation}\label{eq:newLpRellichE}
\frac{|n+p'\al|^p}{(p')^p} \int_{\mathbb R^n} \frac{|\pa_r f|^p}{|x|^{p(1+\al)}} dx \leq \int_{\mathcal R^n} \frac{|\De_r f|^p}{|x|^{p\al}} dx,
\end{equation}
for any $\al \in \R$. The inequalities \eqref{eq:HRLpevenE}, \eqref{eq:HRLpoddE} and \eqref{eq:newLpRellichE} seem to be new in the Euclidean space $\R^n$.

Finally, Theorem \ref{criticalRellich} and  Theorem \ref{criticalhigher} give some new critical Hardy--Rellich type inequalities in $\R^n$ as follows
\begin{equation}\label{eq:criticalRellichevenE}
\lt(\frac{2^{k-1}(k-1)!}{p'}\prod_{i=0}^{k-1} (n-2i-2)\rt)^p \sup_{R >0}\int_{\mathbb R^n} \frac{|f-f_R|^p}{|x|^n \lt|\ln \frac R{|x|}\rt|^p} dx \leq \int_{\mathbb R^n} \frac{|\De_r^k f|^p}{|x|^{n-2kp}} dx,
\end{equation}
for any $1\leq k < n/2$ and any $p>1$, and 
\begin{equation}\label{eq:criticalRellichoddE}
\lt(\frac{2^k k!}{p'}\prod_{i=0}^{k-1} (n-2i-2)\rt)^p \sup_{R >0}\int_{\mathbb R^n} \frac{|f-f_R|^p}{|x|^n \lt|\ln \frac R{|x|}\rt|^p} dx \leq \int_{\mathbb R^n} \frac{|\pa_r \De_r^k f|^p}{|x|^{n-(2k+1)p}} dx,
\end{equation}
for any $0\leq k < (n-2)/2$ and any $p>1$.

The inequality \eqref{eq:criticalRellichoddE} was first proved by Ioku, Ishiwata and Ozawa \cite{IIO2016} in the case $k=0$ and $p=n$ and then was extended by Ruzhansky and Suragan \cite{RS2016} for any $p\in (1,\infty)$. Especially, consider the case $p=2$ then \eqref{eq:criticalRellichevenE} and \eqref{eq:criticalRellichoddE} become
\begin{equation}\label{eq:cevenE}
\lt(2^{2k-2}(2k-1)!\rt)^2 \sup_{R >0}\int_{\mathbb R^n} \frac{|f-f_R|^2}{|x|^n \lt|\ln \frac R{|x|}\rt|^2} dx \leq \int_{\mathbb R^n} |\De_r^k f|^2 dx,
\end{equation}
if $n=4k$ and 
\begin{equation}\label{eq:coddE}
\lt(2^{2k-1} (2k)!\rt)^2 \sup_{R >0}\int_{\mathbb R^n} \frac{|f-f_R|^2}{|x|^n \lt|\ln \frac R{|x|}\rt|^2} dx \leq \int_{\mathbb R^n} |\pa_r \De_r^k f|^2 dx,
\end{equation}
if $n=4k+ 2$, respectively. As consequences of \eqref{eq:cevenE}, \eqref{eq:coddE}, \eqref{eq:MOW} and \eqref{eq:compare}, we obtain the following critical Rellich inequality in $\R^4$ and $\R^6$,
\begin{equation}\label{eq:dim4a}
\sup_{R >0}\int_{\mathbb R^4} \frac{|f-f_R|^2}{|x|^4 \lt|\ln \frac R{|x|}\rt|^2} dx \leq \int_{\mathbb R^4} |\De f|^2 dx,\quad\forall \, f \in C_0^\infty(\R^4)
\end{equation}
and 
\begin{equation}\label{eq:dim6}
16\sup_{R >0}\int_{\mathbb R^6} \frac{|f-f_R|^2}{|x|^6 \lt|\ln \frac R{|x|}\rt|^2} dx \leq \int_{\mathbb R^6} |\na \De f|^2 dx,\quad\forall \, f \in C_0^\infty(\R^6)
\end{equation}

Let us give some comments on our critical Rellich inequalities comparing with the ones of Adimurthi and Santra, i.e. \eqref{eq:ASLp}, \eqref{eq:ASL2}, \eqref{eq:ASchan} and \eqref{eq:ASle}. Note that the inequality \eqref{eq:ASL2} can be extended to any function $ f\in H_0^2(\Om)$ as mentioned in \cite{AS2009}. Clearly, this inequality \eqref{eq:ASL2} is implied from our inequality \eqref{eq:dim4a}. For the case $n=4k$, $k\geq 2$, it can be checked that 
\[
2^{2k-2} (2k-1)! \geq \frac{n}4 \frac1{2^{2k-2}}\prod_{i=0}^{k-2}(4i+2)(8k-4i-6)= k\prod_{i=0}^{k-2} (2i+1) (4k-2i-3).
\]
Hence \eqref{eq:cevenE} improves the inequality \eqref{eq:ASchan} with $n=4k$, $k \geq 2$ since for any radial function $f$ it holds $\De^k f = \De_r^k f$. For the case $n = 4k+2$, $k\geq 1$ we can check that
\[
\frac{2k+1}2 \prod_{i=0}^{k-1} (2k-2i-1) (2k+2i+1) > 2^{2k-1} (2k)!.
\]
Hence the inequality \eqref{eq:ASle} seems to be stronger than the inequality \eqref{eq:coddE}. However, by testing the function $f_\de(x) = (\ln(R/|x|))^{1/2 -\de} \phi(x)$ where $\phi$ is a radial function in $C^\infty_0(\R^n)$ such that $0\leq \phi\leq 1$, $\phi(x) =1$ if $|x|\leq 1$ and $\phi(x) =0$ if $|x|\geq 2$, and $\de >0$ small enough and $R >2$, we see that the inequality \eqref{eq:coddE} is sharp. Thus, our inequality \eqref{eq:coddE} corrects the inequality \eqref{eq:ASle} of Adimurthi and Santra. Moreover, the inequality \eqref{eq:coddE} extends the inequality \eqref{eq:ASle} to all functions (not need to be radial) in the case $n =6$ by \eqref{eq:dim6}.

\section*{Acknowledgments}
This work was supported by the CIMI's postdoctoral research fellowship.

\end{document}